\documentclass[11pt,a4paper]{article}
\usepackage[utf8]{inputenc}
\usepackage{amsmath}
\usepackage{amsfonts}
\usepackage{amsthm}
\usepackage{amssymb}
\usepackage{graphicx}
\usepackage{subcaption}
\usepackage[margin=2cm]{geometry}
\usepackage{hyperref}
\usepackage[capitalise]{cleveref}
\usepackage{multirow}
\usepackage{diagbox} 
\usepackage{mathtools}
\usepackage{xcolor}
\usepackage{enumitem}
\usepackage{mathabx}
\usepackage{cases}
\usepackage{stmaryrd}
\usepackage[framemethod=tikz]{mdframed}
\usepackage{rotating}
\usepackage{tabularx}
\usepackage{pifont}

\allowdisplaybreaks

\newtheorem{thm}{Theorem}
\newtheorem{lem}[thm]{Lemma}
\newtheorem{prop}[thm]{Proposition}

\newtheorem{algo}[thm]{Algorithm}

\theoremstyle{definition}

\newtheorem{rmk}[thm]{Remark}

\newcounter{unnumber}

\newtheorem{prfff}[unnumber]{Proof of Lemma \ref{lemma8}}
\newenvironment{proofff}{\prfff\rm}{\hfill{$\blacksquare$}\endprfff}

\newtheorem{prffff}[unnumber]{Proof of Lemma \ref{lemma18}}
\newenvironment{prooffff}{\prffff\rm}{\hfill{$\blacksquare$}\endprffff}

\newtheorem{prfffff}[unnumber]{Proof of Lemma \ref{lemm2:3}}
\newenvironment{proofffff}{\prfffff\rm}{\hfill{$\blacksquare$}\endprfffff}

\setlist[enumerate]{label=$\rm{(\roman*)}$,leftmargin=\parindent,wide, labelwidth=!, labelindent=0pt}
\numberwithin{equation}{section}
\numberwithin{thm}{section}
\numberwithin{table}{section}
\numberwithin{figure}{section}

\newcommand{\sR}{\mathbb{R}}

\newcommand{\sH}{\mathcal{H}}

\newcommand{\sX}{\mathcal{X}}
\newcommand{\sY}{\mathcal{Y}}
\newcommand{\bH}{\boldsymbol{\sH}}

\newcommand{\mysum}{\displaystyle \sum\limits}

\newcommand{\Id}{\mathrm{Id}}
\newcommand{\bO}{\mathcal{O}}
\newcommand{\bN}{\boldsymbol{N}}
\newcommand{\bP}{\boldsymbol{P}}
\newcommand{\bQ}{\boldsymbol{Q}}
\newcommand{\KM}{Krasnosel'ski\u{\i}-Mann }

\newcommand{\E}{\mathcal{E}}

\newcommand{\F}{\mathcal{F}}

\newcommand{\Lag}{\mathcal{L}}

\newcommand{\dist}{{\rm dist}}

\newcommand{\sol}{\mathcal{S}}

\newcommand{\bz}{\boldsymbol{z}}
\newcommand{\bxi}{\boldsymbol{\xi}}

\title{Fast Forward-Backward splitting for monotone inclusions with a convergence rate of the tangent residual of $o(1/k)$}
\author{Radu Ioan Bo\c{t}\footnote{Faculty of Mathematics, University of Vienna, Oskar-Morgenstern-Platz 1, 1090 Vienna, Austria, e-mail: \url{radu.bot@univie.ac.at}. Research partially supported by Austrian Science Fund (FWF), project W 1260 and P 29809-N32.}
\and Dang-Khoa Nguyen\footnote{Faculty of Mathematics, University of Vienna, Oskar-Morgenstern-Platz 1, 1090 Vienna, Austria, e-mail: \url{dang-khoa.nguyen@univie.ac.at}. Research partially supported by Austrian Science Fund (FWF), project P 29809-N32.} \footnote{Faculty of Mathematics and Computer Science, University of Science, Vietnam National University, Ho Chi Minh City 700000, Vietnam.}
\and Chunxiang Zong\footnote{School of Mathematics and Statistics, Lanzhou University, Lanzhou 730000, People's Republic of China, e-mail: \url{zongchx19@lzu.edu.cn}. Research partially supported by the China Scholarship Council (202206180023) and conducted during a research visit at the University of Vienna in 2023.}}

\begin{document}
	
\maketitle	

\begin{abstract}
We address the problem of finding the zeros of the sum of a maximally monotone operator and a cocoercive operator. Our approach introduces a modification to the forward-backward method by integrating an inertial/momentum term alongside a correction term. We demonstrate that the sequence of iterations thus generated converges weakly towards a solution for the monotone inclusion problem. Furthermore, our analysis reveals an outstanding attribute of our algorithm: it displays rates of convergence of the order $o(1/k)$ for the discrete velocity and the tangent residual approaching zero. These rates for tangent residuals can be extended to fixed-point residuals frequently discussed in the existing literature. Specifically, when applied to minimize a nonsmooth convex function subject to linear constraints, our method evolves into a primal-dual full splitting algorithm. Notably, alongside the convergence of iterates, this algorithm possesses a remarkable characteristic of nonergodic/last iterate $o(1/k)$ convergence rates for both the function value and the feasibility measure. Our algorithm showcases the most advanced convergence and convergence rate outcomes among primal-dual full splitting algorithms when minimizing nonsmooth convex functions with linear constraints.
\end{abstract}	

\noindent \textbf{Key Words.} 
monotone inclusion, 
forward-backward algorithm,
Nesterov's momentum,
Lyapunov analysis,
convergence rates,
convergence of the iterates, convex optimization subject to linear constraints,
primal-dual algorithm
\vspace{1ex}

\noindent \textbf{AMS subject classification.} 49M29, 65K05, 68Q25, 90C25, 65B99
	
\section{Introduction and related works}\label{sec1}	

\subsection{Forward-Backward splitting}\label{subsec11}

A multitude of problems in optimization, variational inequalities, partial differential equations, mechanics, economics, game theory, and signal and image processing can be streamlined by solving a monotone inclusion. Monotone inclusions seeking to determine a zero of the sum of two maximally monotone operators are efficiently solved using the splitting paradigm. This approach involves evaluating the two operators separately from each other \cite{bauschkebook2011,RyuYin}. Depending on the nature of the operators, different splitting approaches have been considered. These include \emph{Douglas-Rachford splitting} \cite{Lions1979Mercier,Combettes2007Pesquet}, \emph{Forward-Backward splitting} \cite{Lions1979Mercier,Passty1979}, \emph{Tseng's Forward-Backward-Forward splitting} \cite{Tseng2000}, and {\it Forward-Reflected-Backward splitting} \cite{Malitsky2020Tam}.

In this paper, our focal point is the \emph{Forward-Backward splitting} approach, with our aim being to notably improve its convergence rate while maintaining its inherent convergence properties.

To introduce the problem, let $\sH$ be a real Hilbert space, $M \colon \sH \to 2^\sH$ a maximally monotone operator and $C \colon \sH \to \sH$ a $\beta$-cocoercive operator with parameter $\beta > 0$, that is
\begin{equation*}
	\left\langle C \left( z \right) - C \left( y \right) , z - y \right\rangle \geq \beta \left\lVert C \left( z \right) - C \left( y \right) \right\rVert ^{2} \quad \forall z , y \in \sH .
\end{equation*}
We aim to solve the monotone inclusion problem
\begin{equation}
\label{intro:pb:eq}
0\in M \left( z \right)+C \left( z \right),
\end{equation}
and assume that it has a solution, namely,
\begin{equation*}
{\cal Z} := \textrm{zer}\left( M + C \right) = \left\lbrace z \in \sH \colon M \left( z \right)+C \left( z \right) \right\rbrace \neq \emptyset. 
\end{equation*}
The \emph{Forward-Backward splitting} (FBS) algorithm, introduced by Passty \cite{Passty1979}, stands out as the most successful and widely used method for solving \eqref{intro:pb:eq}. In its classical formulation, given an initial point $z_{0} \in \sH$, the algorithm generates a sequence $(z_k)_{k \geq 0}$ using the following iteration
\begin{equation}\label{fbs}
(\forall k \geq 0) \quad z_{k+1} := J_{\gamma M} \left(z_k - \gamma C \left( z_{k} \right) \right),
\end{equation}
where $\gamma \in(0,2\beta)$ represents a fixed step size. Here, we define
\begin{equation*}
J_{\gamma M} \colon \sH \to \sH, \quad	J_{\gamma M} := \left( \Id + \gamma M \right) ^{-1} ,
\end{equation*}
as the \emph{resolvent} of the operator $M$ with parameter $\gamma > 0$, where  $\Id \colon \sH \to \sH$ represents the \emph{identity operator}.

The FBS algorithm leverages the specific characteristics of each operator and treats them accordingly. As $C$ is single-valued, it is naturally addressed through a \emph{forward step}. Conversely, the maximal monotonicity of $M$ ensures that its resolvent $J_{\gamma M}$ is a single-valued operator, universally defined, and $1$-cocoercive and is therefore addresses through a \emph{backward step}. However, closed-form expressions for the resolvents of certain operator classes arising in applications are readily available. Regarding its convergence properties, the sequence $(z_k)_{k \geq 0}$ weakly converges to a solution of \eqref{intro:pb:eq} as $k \rightarrow +\infty$. The FBS algorithm \eqref{fbs} and its variations have been extensively explored in \cite{Combettes2005Wajs,Chen1997Rockafellar,Combettes2014Vu,Combettes2015Yamada}.

The rate of convergence of the FBS algorithm is typically assessed using quantities such as $$\left\lVert z_{k} - z_{k-1} \right\rVert \quad \textrm{ and } \quad r_{\mathrm{fix}, \gamma} \left( z_{k} \right) := \left\lVert z_{k} - J_{\gamma M} \left( z_{k} - \gamma C \left( z_{k} \right) \right) \right\rVert.$$ These represent the \emph{discrete velocity} and the \emph{fixed-point residual}, respectively. While these formulas coincide in the context of the classical FBS algorithm \eqref{fbs}, they may differ for its various variants. However, the notable significance attributed to the fixed-point residual is motivated by the fact that the monotone inclusion problem \eqref{intro:pb:eq} can be expressed equivalently as a fixed-point problem. Precisely, for any $\gamma > 0$, the equivalence holds
\begin{equation}\label{pb:fix}
	0 \in M \left( z \right) + C \left( z \right)
	\Leftrightarrow z - \gamma C \left( z \right) \in z + \gamma M \left( z \right) \Leftrightarrow   z =  J_{\gamma M} \left( z - \gamma C \left( z \right) \right) .
\end{equation}
This demonstrates that the FBS algorithm can be viewed as an application of the \KM method \cite{bauschkebook2011,RyuYin} to a specific fixed-point problem. Consequently, the convergence rate of the FBS method can be derived from the more general \KM iteration \cite{Baillon-Bruck,Bravo-Cominetti,Davis-Yin:2016,Matsushita}, and therefore it can be expressed as
\begin{equation*}
	r_{\mathrm{fix}, \gamma} \left( z_{k} \right) = \left\lVert z_{k} - J_{\gamma M} \left( z_{k} - \gamma C \left( z_{k} \right) \right) \right\rVert = o \left( \dfrac{1}{\sqrt{k}} \right) \textrm{ as } k \to + \infty.
\end{equation*}

\subsection{Related works}\label{subsec12}

Over the past two decades, there has been a considerable surge of interest in the optimization community and related fields, focusing on accelerating the convergence speed of first-order methods. It was early recognized that inertial/momentum effects possess the capability to expedite convergence without significantly increasing computational effort \cite{Nesterov1983, Nesterov2004, Nesterov2005, Beck2009Teboulle, Chambolle2016Dossal, Attouch2016Peypouquet, Attouch2018Cabot}. 

In the context of minimizing a differentiable function $\phi \colon \sH \to \sR$, Polyak introduced the \emph{Heavy Ball Method} in \cite{Polyak1964}, given by
\begin{equation}
	\label{algo:Polyak}
	(\forall k \geq 1) \quad z_{k+1} := z_{k} + \alpha \left( z_{k} - z_{k-1} \right) - \gamma \nabla \phi \left( z_{k} \right) ,
\end{equation}
which originates from the discretization of the second-order dynamical system
\begin{equation*}
	\ddot{z} \left( t \right) + \mu \dot{z} \left( t \right) + \nabla \phi \left( z \left( t \right) \right) = 0.
\end{equation*}
Subsequent extensive research on this dynamical system was conducted by Attouch, Goudou, and Redont in \cite{Attouch2000}.

In \cite{Alvarez2000}, Alvarez examined the implicit counterpart of \eqref{algo:Polyak} to address situations where $\phi$ might lack differentiability. Subsequently, Alvarez, in collaboration with Attouch \cite{Alvarez2001Attouch, Alvarez2003},  extended this approach to solve monotone inclusions $0 \in M(z)$, where $M \colon \sH \to 2^\sH$ represents a maximally monotone operator. Given the initial points $z_{0}, z_{1} \in \sH$ and a step size $\gamma > 0$, they proposed the iterative process
\begin{equation*}
(\forall k \geq 1) \quad z_{k+1} - z_{k} - \alpha_{k} \left( z_{k} - z_{k-1} \right) + \gamma M \left( z_{k+1} \right) \ni 0,
\end{equation*}
called \emph{inertial Proximal Point Method} which, in terms of the resolvent, can be equivalently written as
\begin{equation}\label{l4}
(\forall k \geq 1) \quad z_{k+1} := J_{\gamma M} \left( z_{k} + \alpha_{k} \left( z_{k} - z_{k-1} \right) \right).
\end{equation}
The authors demonstrated that, subject to suitable conditions on $\left( \alpha_{k} \right) _{k \geq 1} \in [0,1)$, the sequence $\left( z_{k} \right) _{k \geq 0}$ converges weakly to a zero of $M$ as $k \to + \infty$.

Over the past decade, various approaches have emerged in the literature focusing on integrating inertial effects into the resolution of monotone inclusions with additive structure, as represented in \eqref{intro:pb:eq}. In \cite{Moudafi2003}, Moudafi and Oliny introduced the following iterative scheme
\begin{equation*}
(\forall k \geq 1) \quad 	z_{k+1} := J_{\gamma M} \left( z_{k} + \alpha_{k} \left( z_{k} - z_{k-1} \right) - \gamma C \left( z_{k} \right) \right).
\end{equation*}
Inspired by algorithms designed for optimization problems involving inertial effects 
(see \cite{Nesterov1983,Nesterov2004,Nesterov2005,Beck2009Teboulle}), Lorenz and Pock intoduced the  \emph{inertial Forward-Backward splitting} algorithm in \cite{Lorenz2015Pock}, expressed by
\begin{equation}
(\forall k \geq 1) \quad 	 \left\{
\begin{aligned}\label{lef4}
y_{k} 	& := z_{k}+\alpha_{k}(z_{k}-z_{k-1});  \\
z_{k+1} & := J_{\gamma M}\left(y_{k}-\gamma C \left( y_{k} \right) \right).
\end{aligned}\right.
\end{equation}
Although convergence of the iterates was demonstrated, no convergence rate results could be provided. Moreover, the convergence of iterates was achieved under an assumption regarding the sequence $\left( \alpha_{k} \right) _{k \geq 1}$, necessitating either an adaptive selection dependent on the previously generated iterates or an upper bound constraint of $1/3$.

The \emph{Relaxed Inertial Forward-Backward splitting} algorithm, introduced by Attouch and Cabot in \cite{Attouch2019Cabot} to address \eqref{intro:pb:eq}, is defined as follows:
\begin{equation}
(\forall k \geq 1) \quad \left\{
\begin{aligned}\label{leff4}
y_{k} & :=z_{k}+\alpha_{k}(z_{k}-z_{k-1});  \\
z_{k+1} & :=(1-\rho_k)y_k + \rho_kJ_{\mu_k M}\left(y_{k}-\mu_{k}C(y_{k})\right),
\end{aligned}\right.
\end{equation}
where $\left( \alpha_{k} \right) _{k \geq 1}, \left( \rho_{k} \right) _{k \geq 1}$, and $\left( \mu_{k} \right) _{k \geq 1}$ are positive sequences.  In \cite{Attouch2019Peypouquet}, Attouch and Peypouquet demonstrated that the incorporation of relaxation parameters $\left( \rho_{k} \right)_{k \geq 1}$ in the iterative scheme allows for a more flexible allows for a more flexible selection of the inertial parameters  $\left( \alpha_{k} \right) _{k \geq 1}$ than the previously described.

Driven by their remarkable success in optimization, methods aimed at accelerating the convergence speed of algorithms for monotone inclusions have awakened the interest of the scientific community in recent years. Using the performance estimate problem technique, Kim derived  in \cite{Kim2021} an \emph{Accelerated Proximal Point Method} (APPM) tailored for finding a zero of a maximally monotone operator within a finite-dimensional setting. Given the initial points $y_{-1} = y_{0} = z_{0} = z_{1} \in \sH$, the algorithm is expressed as
\begin{equation}
(\forall k \geq 1) \quad  \left\{
\begin{aligned}\label{Kim}
	y_{k} 	& := z_{k} + \alpha_{k}(z_{k}-z_{k-1}) + \alpha_{k} \left( y_{k-2} - z_{k-1} \right) ;  \\
	z_{k+1} & := J_{\gamma M}\left(y_{k}\right).
\end{aligned}\right.
\end{equation}
This algorithm enhances the first update rule of \eqref{l4} with the \emph{correction term} ``$y_{k-2} - z_{k-1}$". For $\alpha_{k} = 1 - \frac{2}{k+1}$ for every $k \geq 1$, \eqref{Kim} displays a convergence rate for the fixed-point residual
\begin{equation*}
	\left\lVert y_{k} - J_{\gamma M}\left(y_{k}\right) \right\rVert = \bO \left( \dfrac{1}{k} \right) \textrm{ as } k \to + \infty .
\end{equation*}
The convergence of the iterates generated by \eqref{l4} was not addressed in this work.

Recently,  in \cite{Paul2021Emile}, Maing\'{e} introduced the \emph{Corrected Relaxed Inertial Forward-Backward Algorithm} (CRIFBA), achieved by integrating in \eqref{fbs} a constant relaxation factor, an inertial/momentum term and a correction term.  With initial points $y_{0}, z_{0}, z_{1} \in \sH$, the iterative scheme is expressed as
\begin{equation}
(\forall k \geq 1) \quad  \left\{
\begin{aligned}\label{lefff4}
 y_{k} & :=z_{k}+\alpha_{k}(z_{k}-z_{k-1})+\delta_k(y_{k-1}-z_{k});  \\
z_{k+1}& :=(1-\rho)y_k + \rho J_{\gamma M}\left(y_{k}-\gamma C(y_{k})\right),
\end{aligned}\right.
\end{equation}
where $\left( \alpha_{k} \right) _{k \geq 1} \subseteq [0,1)$, $\left( \delta_k \right) _{k \geq 1} \subseteq [0,1)$, $\rho \in (0,1)$, and $\gamma>0$.  Maing\'{e}'s algorithm incorporates a correction term, denoted as ``$y_{k-1}-z_{k}$'', which carries a coefficient that may be distinct from that of the momentum term.  Under conditions where $\alpha_{k} \sim 1 - \frac{1}{k}$ and $\delta_k \sim 1 - \frac{1}{k}$, the weak convergence of the iterate sequence to a zero of $M+C$ is demonstrated. Moreover, convergence rates of $o(1/k)$ are established for both the discrete velocity and the fixed-point residual. The selection of the sequence of inertial parameters $\left( \alpha_{k} \right) _{k \geq 1}$ aligns with those used in fast numerical methods for convex optimization  (see \cite{Chambolle2016Dossal, Attouch2016Peypouquet, Attouch2018Cabot}). 

As \eqref{intro:pb:eq} can be reformulated as a fixed-point problem (see \eqref{pb:fix}), accelerated Forward-Backward splitting methods can be directly derived from techniques designed for fixed-point problems. One such technique, aiming to expedite the convergence of fixed-point problems, traces its origins to the Halpern iteration \cite{Halpern} and involves incorporating an anchoring point \cite{Yoon2021Ryu,Tran-Dinh-Lou,Park2022Ryu,Yang2022Zheng}. In \cite{Tran-Dinh}, Tran-Dinh highlighted connections between this approach and methods utilizing an inertial/momentum term, particularly for specific parameter choices.

Recently, in \cite{Bot2023Nguyen}, we derived an iterative scheme for solving \eqref{intro:pb:eq} from the general \emph{Fast Krasnosel'ski\u{\i}-Mann algorithm} (fKM) introduced for fixed-point problems (see also \cite{Bot2023Robert_Nguyen}). With initial points  $z_{0}, z_{1} \in \sH$, $\alpha >2$, step size $0 < \gamma < 2 \beta$ and $0 < s < 2 - \frac{\gamma}{2 \beta}$, the iterative scheme is defined as
\begin{align} \label{fast_KM}
(\forall k \geq 1) \quad z_{k+1} := & \ \left( 1 - \dfrac{s\alpha}{2 \left( k + \alpha \right)} \right) z_{k}  + \dfrac{(1-s)k}{k+\alpha} (z_k - z_{k-1}) + \dfrac{s(\alpha + 2k)}{2(k + \alpha)} J_{\gamma M} \left( z_{k} - \gamma C \left( z_{k} \right) \right) \nonumber 
\\
& \ - \dfrac{sk}{k + \alpha} J_{\gamma M} \left( z_{k-1} - \gamma C \left( z_{k-1} \right) \right) .
\end{align}
The resulting iterates converge weakly to a solution of \eqref{intro:pb:eq}. Furthermore, the algorithm exhibits a convergence rate of $o \left( \frac{1}{k } \right)$ for both the discrete velocity $\left\lVert z_k -z_{k-1} \right\rVert$ and the fixed-point residual $r_{\mathrm{fix}, \gamma} \left( z_{k} \right)$ as $k \rightarrow +\infty$. 

\subsection{Our contributions}\label{subsec13}

We propose a \emph{Fast Forward-Backward algorithm} to solve \eqref{intro:pb:eq}, incorporating an inertial/momentum term alongside a correction term. For the sequence $(z_k)_{k \geq 0}$ generated by our method, we demonstrate that both the discrete velocity $\left\lVert z_k - z_{k-1} \right\rVert$ and the tangent residual $r_{\mathrm{tan}} \left( z_k \right)$ converge to zero with a rate of $o \left( \frac{1}{k} \right)$ as $k \rightarrow +\infty$. The \emph{tangent residual} is defined as
\begin{equation*}
	r_{\mathrm{tan}} \left( z \right) := \dist (0, M(z) + C(z)) = \inf_{\xi\in M(z)} \left\lVert \xi + C \left( z \right) \right\rVert ,
\end{equation*}
and it holds (see Lemma \ref{bb1} in the Appendix) that
$$r_{\mathrm{fix}, \gamma} \left( z \right) =  \left\lVert z - J_{\gamma M} \left( z - \gamma C \left( z \right) \right) \right\rVert \leq \gamma \dist (0, M(z) + C(z))  = \gamma r_{\mathrm{tan}} \left( z \right) \quad \forall z \in \sH.$$
For discussions concerning these convegence measures in the context of \emph{variational inequality} problems, we refer to \cite{Sedlmayer2023}.

Hence, the convergence rate of $o \left( \frac{1}{k} \right)$ as $k \rightarrow +\infty$ for the fixed-point residual immediately follows. Moreover, we will establish the weak convergence of the sequence $(z_k)_{k \geq 0}$ to a solution of \eqref{intro:pb:eq} as $k \rightarrow +\infty$.

As a significant application case, we derive a primal-dual full splitting algorithm for solving optimization problems of the form
\begin{align}
\label{intro:mn-mn}
	&\min\limits_{x \in \sX}  f\left(x \right) + h\left(x \right),\nonumber \\
& \textnormal{subject to } Ax=b.
\end{align}
Here, $\sX$ and $\sY$ represent real Hilbert spaces, where $f \colon \sX \to \mathbb{R} \cup \{+\infty\}$ is a proper,  convex and lower semicontinuous function, $h\colon \sX \to \mathbb{R}$ is a convex and differentiable function with a gradient that is $\beta^{-1}$-Lipschitz continuous ($\beta > 0$),  and $A \colon \sX\rightarrow \sY$ is a linear continuous operator. The primal-dual algorithm is derived from the general iterative scheme via the product space approach by  V{\~{u}} \cite{Vu2011}. 

Given the generated sequence of primal-dual iterates $\left(x_{k},\lambda_{k} \right)_{k \geq 0}$, we observe convergence rates for the velocities
\begin{equation*}
	\left\lVert x_{k} - x_{k-1} \right\rVert = o \left( \dfrac{1}{k} \right)	
	\qquad \textrm{ and } \qquad
	\left\lVert \lambda_{k} - \lambda_{k-1} \right\rVert = o \left( \dfrac{1}{k} \right)
	\qquad \textrm{ and } \qquad
	k \to + \infty.
\end{equation*}
Moreover, we derive convergence rates for the \emph{function values} and the \emph{feasibility measure}:
\begin{equation*}
	\left\lvert \left( f+h \right) \left( x_{k} \right) - \left( f+h \right) \left( x_{*} \right) \right\rvert = o \left( \dfrac{1}{k} \right)
	\qquad \textrm{ and } \qquad
	\left\lVert Ax_{k} - b \right\rVert = o \left( \dfrac{1}{k} \right) 
	\qquad \textrm{ and } \qquad
	k \to + \infty,
\end{equation*}
respectively. Here, $x_{*}$ denotes an optimal solution of \eqref{intro:mn-mn} and  $\left( f+h \right) \left( x_{*} \right)$  its optimal objective value. The convergence results are derived from the rates of convergence for the tangent residual in the monotone inclusion problem.   It is important to highlight that deriving such statements from the convergence rate results based on fixed-point residual seems unattainable. Additionally, we establish the weak convergence of $\left(x_{k},\lambda_{k} \right)_{k \geq 0}$ to a primal-dual optimal solution.

Finally, we verify the theoretical findings by assessing the performance of the proposed algorithm through numerical experiments.

\section{The proposed algorithm}\label{sec2}

In this section, we formulate the algorithm and conduct a thorough analysis of its convergence.

\begin{mdframed}
\begin{algo}
\label{algo:ffb}
Let $\alpha > 2, \gamma > 0$, $z_{0}, y_{0} \in \sH$, and $z_{1} := J_{\gamma M} \left( y_{0} - \gamma C \left( z_{0} \right) \right)$. We set
\begin{equation}
\label{algo:z-y}
(\forall k \geq 1) \quad
\begin{dcases}
y_{k} 
& := z_{k} + \left( 1 - \dfrac{\alpha}{k + \alpha} \right) \left( z_{k} - z_{k-1} \right) + \left( 1 - \dfrac{\alpha}{2 \left( k + \alpha \right)} \right) \left( y_{k-1} - z_{k} \right);  \\
z_{k+1}
& := J_{\gamma M} \left( y_{k}-\gamma C \left( z_{k} \right)\right).
\end{dcases}
\end{equation}
\end{algo}
\end{mdframed}

In the following, we provide an equivalent formulation for \cref{algo:ffb}, which will play a central role in the convergence analysis. 

\begin{prop}
\label{rmk:sub}
Let $z_{0}, y_{0} \in \sH$, $z_{1} := J_{\gamma M} \left( y_{0} - \gamma C \left( z_{0} \right) \right)$, and $\xi_{1} := \frac{1}{\gamma} \left( y_{0} - z_{1} \right) - C \left( z_{0} \right) \in M \left( z_{1} \right)$. Then the sequence $(z_k)_{k \geq 0}$ generated in \cref{algo:ffb} can also be generated equivalently by the following iterative scheme
\begin{equation}
\label{algo:z-xi}
(\forall k \geq 1) \quad \!\!\!
\begin{dcases}
z_{k+1} \!\!
& := J_{\gamma M} \left( z_{k} -\gamma C \left( z_{k} \right) + \left( 1 - \dfrac{\alpha}{k + \alpha} \right) \left( z_{k} - z_{k-1} \right) + \frac{2k+\alpha}{2(k+\alpha)} \gamma \left(\xi_k +C \left( z_{k-1} \right)\right)\right);\\
\xi_{k+1} \!\!
& := \frac{1}{\gamma}\left( z_{k} -z_{k+1} + \left( 1 - \dfrac{\alpha}{k + \alpha} \right) ( z_{k} - z_{k-1} )\right) + \frac{2k+\alpha}{2(k+\alpha)} \left(\xi_k +C \left( z_{k-1} \right)\right) - C \left( z_{k} \right) .
\end{dcases}
\end{equation}
In addition, it holds
\begin{equation}
\label{algo:inc-M}
\xi_{k}\in M \left( z_{k} \right) \qquad \quad \forall k\geq 1.
\end{equation}
\end{prop}

\begin{proof}
Given $z_{0}, y_{0} \in \sH$ and $z_{1} := J_{\gamma M} \left( y_{0} - \gamma C \left( z_{0} \right) \right)$,
it follows that
\begin{equation*}
\Leftrightarrow \xi_{1} = \frac{1}{\gamma} \left( y_{0} - z_{1} \right) - C \left( z_{0} \right) \in M \left( z_{1} \right) .
\end{equation*}
In the same way, by invoking also the second update block of  \cref{algo:ffb}, we obtain
\begin{equation*}
\xi_{k+1} := \frac{1}{\gamma}\left( y_{k}-z_{k+1}\right) - C \left( z_{k} \right) \in M(z_{k+1}) \quad \forall k \geq 0.
\end{equation*}
Therefore, the iterative scheme in \cref{algo:ffb} can be equivalently written as
\begin{equation}\
(\forall k \geq 1) \quad \left\{
\begin{aligned}\label{ouralg1}	
y_{k} 		& := z_{k} + \left( 1 - \dfrac{\alpha}{k + \alpha} \right) \left( z_{k} - z_{k-1} \right) + \left( 1 - \frac{\alpha}{2(k+\alpha)} \right) \gamma \left(\xi_k +C \left( z_{k-1} \right)\right);  \\
z_{k+1} 	& :=  J_{\gamma M} \left( y_{k}-\gamma C \left( z_{k} \right)\right);\\
\xi_{k+1}	& := \frac{1}{\gamma}\left( y_{k}-z_{k+1}\right) - C \left( z_{k} \right) ,
\end{aligned}\right.
\end{equation}
which, after some simplifications, it transforms into \eqref{algo:z-xi}.

Conversely, starting from \eqref{algo:z-xi}, we can define a sequence $\left( y_{k} \right)_{k \geq 1}$ following the first update block in  \eqref{ouralg1}. Consequently, for every $k \geq 1$, the relation
\begin{equation}
	\gamma \left( \xi_{k+1} +C \left( z_{k} \right) \right) = y_k-z_{k+1} ,
\end{equation}
holds, leading to the transformation of \eqref{algo:z-xi} into \eqref{algo:z-y}.
\end{proof}

\begin{rmk}\label{rem23}
\begin{enumerate}
\item Remarkably, within the first update block of \eqref{algo:z-y}, we incorporate both a \emph{momentum term} denoted as ``$z_{k} - z_{k-1}$" and a \emph{correction term} denoted as ``$y_{k-1} - z_{k}$".  Nevertheless, in the forward-backward step, our method evaluates $C$ at $z_k$, thereby it can be seen as an extension of \cite{Moudafi2003}.  Notably, it differes significantly from other forward-backward schemes in the existing literature.

\item The apparent similarity between \cref{algo:ffb} and the CRIFBA method is misleading. While the initial update block may appear similar, a significant conceptual distinction arises due to our scheme evaluating $C$ at $z_k$, contrasting with the CRIFBA method's evaluation of $C$ at $y_k$. This disparity leads to a substantial difference in how the correction terms are defined, even when opting for $\rho=1$ in the CRFIBA method. Notably, this falls outside the scope of the convergence evidence established in \cite{Paul2021Emile}.
\end{enumerate}
\end{rmk}

The reformulation \eqref{algo:z-xi} reveals some connection with other Nesterov's acceleration-type methods in the literature. This inspires us to consider, for $z_* \in {\cal Z}$, the following discrete energy function $\E_{\eta,k}$ for $0 \leq \eta \leq \alpha-1$ and $k \geq 1$:
\begin{align}
	\E_{\eta,k}
	 := & \ \dfrac{1}{2} \left\lVert 2 \eta \left( z_{k} - z_{*} \right) + 2k \left( z_{k} - z_{k-1} \right) + \frac{5\alpha-2}{4(\alpha-1)}\gamma k \left(\xi_k +C \left( z_{k-1} \right)\right) \right\rVert ^{2} \nonumber\\
&	+ 2 \eta \left( \alpha - 1 - \eta \right) \left\lVert z_{k} - z_{*} \right\rVert ^{2}
	 + 2 \eta \gamma\left(\frac{3(\alpha-2)}{4(\alpha-1)}k+\alpha\right) \big\langle z_{k} - z_{*} , \xi_k +C \left( z_{k-1} \right) \big\rangle \nonumber\\
& + \frac{1}{2}\gamma^{2}\left(\frac{3(\alpha-2)}{4(\alpha-1)} k+\alpha\right)\left( \frac{5\alpha-2}{4(\alpha-1)} k + \alpha \right)  \left\lVert \xi_k +C \left( z_{k-1} \right)\right\rVert ^{2}\nonumber\\
=&\ 2 \eta \left( \alpha - 1 \right) \left\lVert z_{k} - z_{*} \right\rVert ^{2} + 4 \eta k \big\langle z_{k} - z_{*} , z_{k} - z_{k-1} + \gamma (\xi_k+C \left( z_{k-1} \right)) \big\rangle \nonumber \\
	& + 2 \eta \gamma \alpha\langle z_k-z_{*},\xi_k+C \left( z_{k-1} \right)\rangle+\dfrac{1}{2} k^{2} \left\lVert 2 \left( z_{k} - z_{k-1} \right) + \frac{5\alpha-2}{4(\alpha-1)}\gamma (\xi_k+C \left( z_{k-1} \right)) \right\rVert ^{2} \nonumber \\
&+ \frac{1}{2} \gamma^{2}\left(\frac{3(\alpha-2)}{4(\alpha-1)}k +\alpha\right) \left( \frac{5\alpha-2}{4(\alpha-1)} k + \alpha \right) \left\lVert \xi_k +C \left( z_{k-1} \right) \right\rVert ^{2} \label{im:defi:E-k:eq}.
\end{align}

In the upcoming lemma, we introduce a property of the discrete energy sequence $\left( \E_{\eta,k} \right) _{k \geq 1}$, crucial for establishing the convergence results. Its proof can be found in Appendix \ref{subseca2}.

\begin{lem}\label{lemma8}
Let $z_* \in {\cal Z}$, $\left(z_{k} \right)_{k \geq 0}$ be the sequence generated by Algorithm \ref{algo:ffb}, and  $\left(\xi_{k} \right)_{k \geq 1}$ the sequence generated by the iterative scheme \eqref{algo:z-xi}. For every $0 \leq \eta \leq \alpha-1$ and every $k \geq 1 $ it holds
\begin{align}
\E_{\eta,k+1} - \E_{\eta,k}
\leq \ & 2 \eta \gamma(2-\alpha)\left\langle z_{k+1} - z_{*} , \xi_{k+1}+C \left( z_{k} \right) \right\rangle+ 2 \gamma \left(\nu_{0}k +\nu_{1}\right) \left\langle z_{k+1} - z_{k} , \xi_{k+1}+C \left( z_{k} \right) \right\rangle \nonumber \\
& + 2\nu_{2} k \left\lVert z_{k+1} - z_{k} \right\rVert ^{2} +\frac{1}{2} \gamma^{2} (\nu_{3}k +\nu_{4}) \left\lVert \xi_{k+1}+C \left( z_{k} \right) \right\rVert ^{2} \nonumber \\
& - 2 \gamma\left(\frac{3(\alpha-2)}{4(\alpha-1)}k+\alpha\right) \left( k + \alpha \right) \left\langle z_{k+1} - z_{k} , \xi_{k+1}+C \left( z_{k} \right) - (\xi_{k}+C \left( z_{k-1} \right)) \right\rangle\nonumber\\
&- \gamma^{2}\left(\frac{3(\alpha-2)}{4(\alpha-1)}k+\alpha\right)(k+\alpha)  \left\lVert \xi_{k+1}+C \left( z_{k} \right) - (\xi_{k}+C \left( z_{k-1} \right)) \right\rVert ^{2},
\label{dec:inq}
\end{align}
where
\begin{align*}
	\nu_{0} 	& := \frac{3(\alpha-2)}{4(\alpha-1)}\eta+\frac{5\alpha-2}{4(\alpha-1)} \left( \eta + 1 - \alpha \right) -\frac{(4\alpha-1)(\alpha-2)}{4(\alpha-1)} \nonumber\\
	\nu_{1} 	& :=\alpha(\eta+1-\alpha)+(\eta-1)\alpha+\frac{(5\alpha-2)}{4(\alpha-1)}\nonumber\\
    \nu_{2}		& := 2 \left( \eta + 1 - \alpha \right)\leq0\nonumber\\
    \nu_{3}		& :=\frac{5\alpha-2}{2(\alpha-1)}(2-\alpha)< 0 \nonumber\\
    \nu_{4}		& :=\frac{4\alpha^2+\alpha-2}{2(\alpha-1)}-\alpha^2.
\end{align*}
\end{lem}

Further, for $0<\varepsilon<1$ and $0 \leq \eta \leq \alpha-1$ we consider the following regularization of the discrete energy function, defined for every $k \geq 1$ by
\begin{align}
	\F_{\eta,k}
:=& \ \E_{\eta,k}-2 \gamma\left(\frac{3(\alpha-2)}{4(\alpha-1)}k+\alpha \right)k   \left\langle z_{k} - z_{k-1} , C \left( z_{k} \right) - C \left( z_{k-1} \right)\right\rangle
\nonumber\\
&+\frac{1}{2}\gamma^{2}\left(\frac{3(\alpha-2)}{4(\alpha-1)}k+\alpha \right)\left((2-\varepsilon)(2k+\alpha)+\alpha\sqrt{\frac{3(\alpha-2)}{4(\alpha-1)}k+\alpha} \right)  \left\lVert C \left( z_{k} \right)-C \left( z_{k-1} \right) \right\rVert^2\nonumber\\
= & \ \dfrac{1}{2} \left\lVert 2 \eta \left( z_{k} - z_{*} \right) + 2k \left( z_{k} - z_{k-1} \right) + \frac{5\alpha-2}{4(\alpha-1)}\gamma k \left(\xi_k +C \left( z_{k-1} \right)\right) \right\rVert ^{2} \nonumber\\
&	+ 2 \eta \left( \alpha - 1 - \eta \right) \left\lVert z_{k} - z_{*} \right\rVert ^{2}
	 + 2 \eta \gamma\left(\frac{3(\alpha-2)}{4(\alpha-1)}k+\alpha\right) \big\langle z_{k} - z_{*} , \xi_k +C \left( z_{k-1} \right) \big\rangle \nonumber\\
& + \frac{1}{2}\gamma^{2}\left(\frac{3(\alpha-2)}{4(\alpha-1)}k+\alpha\right)\left( \frac{5\alpha-2}{4(\alpha-1)} k + \alpha \right)  \left\lVert \xi_k +C \left( z_{k-1} \right) \right\rVert ^{2}\nonumber\\
&-2 \gamma\left(\frac{3(\alpha-2)}{4(\alpha-1)}k+\alpha\right)k  \left\langle z_{k} - z_{k-1} , C \left( z_{k} \right) - C \left( z_{k-1} \right)\right\rangle
\nonumber\\
&+\frac{1}{2}\gamma^{2}\left(\frac{3(\alpha-2)}{4(\alpha-1)}k+\alpha\right)\left((2-\varepsilon)(2k+\alpha)+\alpha\sqrt{\frac{3(\alpha-2)}{4(\alpha-1)}k+\alpha} \right)  \left\lVert C \left( z_{k} \right)-C \left( z_{k-1} \right) \right\rVert^2.
\label{def:G}
\end{align}
The properties of the regularized discrete energy sequence $\left( \F_{\eta,k} \right)_{k \geq 0}$ are presented in \cref{lemma18}, the proof of which is also provided in Appendix \ref{subseca2}.

\begin{lem}\label{lemma18}
Let $z_* \in {\cal Z}$, $\left(z_{k} \right)_{k \geq 0}$ be the sequence generated by Algorithm \ref{algo:ffb}, and  $\left(\xi_{k} \right)_{k \geq 1}$ the sequence generated by the iterative scheme \eqref{algo:z-xi}. For every $0 \leq \eta \leq \alpha-1$ and every $0<\varepsilon<1$, the following statements are true:
\begin{enumerate}
\item 
for every $k\geq 1$, it holds:
\begin{align}
\F_{\eta,k+1} - \F_{\eta,k}
&\,\,\leq\frac{4(\alpha -2)(\alpha-1)}{3(k+1)\sqrt{k+1}}\eta^2\left\lVert z_{k+1}-z_{*} \right\rVert^2-2 \eta \gamma(\alpha-2)\left\langle z_{k+1} - z_{*} , \xi_{k+1}+C \left( z_{k+1} \right) \right\rangle\nonumber\\
&\quad+ \frac{1}{2}\gamma^{2}\left(\nu_{3} k+\alpha\sqrt{\nu_{5} k+\alpha } +\nu_{4}\right) \left\lVert \xi_{k+1}+C \left( z_{k} \right) \right\rVert^2 + 2\nu_{2} k \left\lVert z_{k+1} - z_{k} \right\rVert ^{2}\nonumber\\
&\quad-\omega_{k} \|C \left( z_{k+1} \right)-C \left( z_{k} \right)\|^2+ 2 \gamma \left(\nu_{0}k +\nu_{1}\right) \left\langle z_{k+1} - z_{k} , \xi_{k+1}+C \left( z_{k} \right) \right\rangle \nonumber \\
&\quad - \gamma^{2}\left(\nu_{5} k+\alpha\right)\left(\nu_{6} k+\nu_{7}\right)\|\xi_{k+1}+C \left( z_{k} \right) - (\xi_{k}+C \left( z_{k-1} \right))\|^2,\label{bbck2}
\end{align}
where
\begin{align*}
\nu_{5} 	:= \frac{3}{4(\alpha-1)}(\alpha-2) >0, \quad
\nu_{6}     := \frac{1}{2-\varepsilon}(1-\varepsilon)>0, \quad
\nu_{7}	:= \frac{1}{2(2-\varepsilon)}(3-2\varepsilon) \alpha>0,
\end{align*}
and
\begin{align}
\omega_{k} 	& := \left(2\beta-(2-\varepsilon) \gamma \right)\frac{3(\alpha -2)}{4(\alpha-1)}\gamma (k+1)^2-\frac{3(\alpha-2)}{4(\alpha-1)}\gamma^{2}(k+1)\sqrt{k+1}\nonumber\\
&\quad \,\,-\frac{1}{2}\alpha \gamma^{2}\left(\frac{3(\alpha-2)}{4(\alpha-1)}(k+1)+\alpha\right)\sqrt{\frac{3(\alpha-2)}{4(\alpha-1)}(k+1)+\alpha}-\frac{1}{2}(2-\varepsilon)\gamma^{2}\alpha^2\nonumber\\
&\quad \,\,+\left(2\beta-(2-\varepsilon)\frac{11\alpha-14}{8(\alpha-1)} \gamma \right)\alpha \gamma (k+1). \label{defi:mu-k}
\end{align}
\item 
for every $k\geq 1$, one has the following lower bound for the quantity $\F_{\eta,k}$
\begin{align}
\F_{\eta,k}
\geq &\ \frac{3(\alpha-2)}{4(5\alpha-2)} \left\lVert 4 \eta \left( z_{k} - z_{*} \right)+ 2k \left( z_{k} - z_{k-1} \right) + \frac{5\alpha-2}{2(\alpha-1)}\gamma k \left(\xi_k +C \left( z_{k-1} \right)\right) \right\rVert ^{2} \nonumber\\
&	+ 2 \eta(\alpha-1) \left( 1 - \frac{8\eta}{5\alpha-2} \right) \left\lVert z_{k} - z_{*} \right\rVert ^{2}
+ \frac{1}{2} \gamma^{2}\alpha (2k+\alpha) \left\lVert \xi_k +C \left( z_{k-1} \right) \right\rVert ^{2}\nonumber\\
&+\left(3(\alpha-2)\left(\frac{1}{5\alpha-2}-\frac{1}{4(2-\varepsilon)(\alpha-1)}\right) k^2-\frac{2}{\beta}\alpha \gamma k\right) \left\lVert z_k-z_{k-1} \right\rVert ^{2}\nonumber\\
&+2\eta \alpha \gamma \big\langle z_{k} - z_{*} , \xi_k +C \left( z_{k-1} \right) \big\rangle.
\label{G:inequa}
\end{align}
\end{enumerate}
\end{lem}

The last of the three technical lemmas describe the intervals where the parameters that appear in the formulation of the algorithm and in the definition of the discrete energy function can be chosen in order to obtain convergence for the iterates and to derive convergence rates.  Its proof is also provided in Appendix \ref{subseca2}.

\begin{lem}\label{lemm2:3}
Let $z_* \in {\cal Z}$, $\left(z_{k} \right)_{k \geq 0}$ be the sequence generated by Algorithm \ref{algo:ffb}, and  $\left(\xi_{k} \right)_{k \geq 1}$ the sequence generated by the iterative scheme \eqref{algo:z-xi}. The following statements are true:
\begin{enumerate}
\item 
\label{lemm2:3:S_k}
if $0 < \gamma< \frac{2}{2-\varepsilon} \beta $, $0<\varepsilon<\frac{3(\alpha-2)}{4(\alpha-1)}$ and $0<\eta<\frac{5\alpha-2}{8}$, then there exists an integer $K_{0}(\eta) \geq1$ such that for every $k\geq K_{0}(\eta)$
\begin{equation}\label{dec:mu}
\omega_{k} \geq \left(2\beta-(2-\varepsilon) \gamma \right)\frac{3(\alpha -2)}{8(\alpha-1)}\gamma (k+1)^2, \quad \quad \frac{3}{2}(\alpha-2)\left(\frac{1}{5\alpha-2}-\frac{1}{4(2-\varepsilon)(\alpha-1)}\right) k^2-\frac{2}{\beta}\alpha \gamma k\geq 0,
\end{equation}
and
\begin{equation}\label{def:S0}
S_k:=\eta(\alpha-1)\left(1-\frac{8}{5\alpha-2}\eta\right)\left\lVert z_k-z_* \right\rVert^2+2\eta \alpha \gamma\langle z_k-z_*,\xi_k+C \left( z_{k-1} \right)\rangle+\frac{1}{2}\alpha \gamma^{2} k\left\lVert \xi_k+C \left( z_{k-1} \right) \right\rVert^2 \geq 0,
\end{equation}
which yields, according to \eqref{G:inequa},
\begin{equation}\label{def:S}
\F_{\eta,k} \geq \eta(\alpha-1)\left(1-\frac{8}{5\alpha-2}\eta\right)\left\lVert z_k-z_* \right\rVert^2 \geq 0 ;
\end{equation}

\item 
\label{lemm2:3:R_k}
there exist two parameters
\begin{equation}\label{dec:lambda}
0\leq \underline{\eta}(\alpha)<\overline{\eta}(\alpha)\leq \frac{5\alpha-2}{8}
\end{equation}
such that for every $\eta$ satisfying $\underline{\eta}(\alpha)<\eta<\overline{\eta}(\alpha)$ one can find an integer $K_{1}(\eta)\geq 1$ with the property that the following inequality holds for every $k\geq K_{1}(\eta)$
\begin{align}\label{dec:R}
R_k:= & \ 2\sqrt{\frac{9\alpha-2}{2(5\alpha-2)}}\nu_{2} k \left\lVert z_{k+1} - z_{k} \right\rVert ^{2}+ 2 \gamma \left(\nu_{0}k +\nu_{1}\right) \left\langle z_{k+1} - z_{k} , \xi_{k+1}+C \left( z_{k} \right) \right\rangle \nonumber \\
& \ + \frac{1}{2}\sqrt{\frac{9\alpha-2}{2(5\alpha-2)}}\gamma^{2}\left(\nu_{3} k+\alpha \sqrt{\nu_{5}k+\alpha}+\nu_{4} \right) \left\lVert \xi_{k+1}+C \left( z_{k} \right) \right\rVert^2 \leq0;
\end{align}
\end{enumerate}
\end{lem}

Following the previous preparatory results, we assume that  the step size $\gamma$ satisfies
\begin{mdframed}
$$0 < \gamma < \frac{8(\alpha-1)}{5\alpha-2}\beta.$$
\end{mdframed}
Therefore, 
\begin{equation}\label{epsilon}
\mbox{there exists} \ 0<\varepsilon<\frac{3(\alpha-2)}{4(\alpha-1)}<1 \ \mbox{such that} \ 0 < \gamma< \frac{2}{2-\varepsilon} \beta.
\end{equation}
Given that $\frac{8(\alpha-1)}{5\alpha-2}\beta < 2 \beta$, the condition on the step size is somewhat more restrictive than in the classical forward-backward setting, which typically permits $0 < \gamma < 2 \beta$. However, despite this, our methods demonstrate a superior theoretical rate of convergence when compared to the classical forward-backward algorithm.

\begin{prop}\label{prop:lim}
Let $z_* \in {\cal Z}$, $\left(z_{k} \right)_{k \geq 0}$ be the sequence generated by Algorithm \ref{algo:ffb}, and  $\left(\xi_{k} \right)_{k \geq 1}$ the sequence generated by the iterative scheme \eqref{algo:z-xi}. Then the following statements are true:
\begin{enumerate}
\item 
\label{prop:lim:sum}
it holds 
\begin{subequations}
\label{rate:sum}
\begin{align}		
\mysum_{k \geq 1}  \left\langle z_{k} - z_{*} , \xi_{k}+C \left( z_{k} \right) \right\rangle < + \infty ,  \\
\mysum_{k \geq 1} k^2 \left\lVert C \left( z_{k} \right) - C \left( z_{k-1} \right) \right\rVert ^{2} < + \infty , \\
\mysum_{k \geq 1} k^2 \left\lVert \xi_{k+1}+C \left( z_{k} \right) - \left(\xi_k+C \left( z_{k-1} \right)\right) \right\rVert ^{2} < + \infty , \\
\mysum_{k \geq 1} k \left\lVert z_{k+1} - z_{k} \right\rVert ^{2} < + \infty , \\
\mysum_{k \geq 1} k \left\lVert \xi_{k+1}+C \left( z_{k} \right) \right\rVert ^{2} < + \infty;
\end{align}
\end{subequations}

\item 
\label{prop:lim:rate}
the sequence $(z_k)_{k\geq 0}$ is bounded and 
\begin{align*}
\left\lVert z_k-z_{k-1} \right\rVert = \bO \left( \dfrac{1}{k } \right),\,\quad  \left\lVert \xi_k+C \left( z_{k} \right) \right\rVert = \bO \left( \dfrac{1}{k} \right),
\,\,\quad
\left\langle z_{k} - z_{*} , \xi_k+C \left( z_{k} \right)\right\rangle = \bO \left( \dfrac{1}{k } \right) \quad \mbox{as} \ k \rightarrow + \infty;
\end{align*}

\item 
\label{prop:lim:conv} for every $\eta \in \left( \underline{\eta}(\alpha), \overline{\eta}(\alpha) \right)$, where  $\underline{\eta}(\alpha)$ and $\overline{\eta}(\alpha)$ are the parameters provided by Lemma \ref{lemm2:3} (ii), the sequences $(\E_{\eta,k})_{k\geq 1}$ and $(\F_{\eta,k})_{k\geq 1}$ are convergent and have the same limit.
\end{enumerate}
\end{prop}

\begin{proof}
From the step size assumption we deduce that there exists $0<\varepsilon<\frac{3(\alpha-2)}{4(\alpha-1)}<1$ such that $0 < \gamma< \frac{2}{2-\varepsilon} \beta$. Furthermore, we consider the parameters $\underline{\eta}(\alpha)$ and $\overline{\eta}(\alpha)$ provided by Lemma \ref{lemm2:3} (ii) and choose $\eta \in \left( \underline{\eta}(\alpha), \overline{\eta}(\alpha) \right)$. From \cref{lemm2:3}  (i) we obtain an integer  $K_{0}(\eta) \geq 1$ such that for every $k\geq K_{0}(\eta)$  the inequalities (\ref{dec:mu}), (\ref{def:S0}) and (\ref{def:S}) hold; from \cref{lemm2:3} (ii), we obtain an integer $K_{1}(\eta)\geq 1$ such that for every $k\geq K_{1}(\eta)$ the inequality (\ref{dec:R}) holds.

Invoking \eqref{bbck2}, for every $k\geq K_{2}(\eta):=\max\{K_{0}(\eta), K_{1}(\eta)\}$ it holds
\begin{align}
\F_{\eta,k+1} - \F_{\eta,k} 
\leq & \  \frac{4(\alpha -2)(\alpha-1)}{3(k+1)\sqrt{k+1}}\eta^2\left\lVert z_{k+1}-z_{*} \right\rVert^2-2 \eta \gamma(\alpha-2)\left\langle z_{k+1} - z_{*} , \xi_{k+1}+C \left( z_{k+1} \right) \right\rangle\nonumber\\
& \ + \frac{1}{2}\left(1-\sqrt{\frac{9\alpha-2}{2(5\alpha-2)}}\right)\gamma^{2}\left(\nu_{3} k+\alpha\sqrt{\nu_{5} k+\alpha } +\nu_{4}\right) \left\lVert \xi_{k+1}+C \left( z_{k} \right) \right\rVert^2 \nonumber\\
& \ - \omega_{k} \|C \left( z_{k+1} \right)-C \left( z_{k} \right)\|^2+ 2\left(1-\sqrt{\frac{9\alpha-2}{2(5\alpha-2)}}\right)\nu_{2} k \left\lVert z_{k+1} - z_{k} \right\rVert ^{2} \nonumber\\
& \ - \gamma^{2}\left(\nu_{5} k+\alpha\right)\left(\nu_{6} k+\nu_{7}\right)\|\xi_{k+1}+C \left( z_{k} \right) - (\xi_{k}+C \left( z_{k-1} \right))\|^2.\label{func:G}
\end{align}
We set
\begin{align*}
	\nu_{8} :=\frac{4}{3}(\alpha-2)\eta\left(1-\frac{8}{5\alpha-2}\eta\right)^{-1}>0 \quad \mbox{and} \quad
	\nu_{9} := \left(2\beta-(2-\varepsilon) \gamma \right)\frac{3(\alpha -2)}{8(\alpha-1)}\gamma > 0.
\end{align*}
In view of (\ref{def:S}), we have that for every $k\geq K_{2}(\eta)$
\begin{align*}
\frac{4(\alpha -2)(\alpha-1)}{3(k+1)\sqrt{k+1}}\eta^2\left\lVert z_{k+1}-z_{*} \right\rVert^2 
&\leq \frac{\nu_{8}}{(k+1)\sqrt{k+1}}\F_{\eta,k+1},\nonumber
\end{align*}
therefore
\begin{align}
& \ \left(1-\frac{\nu_{8}}{(k+1)\sqrt{k+1}}\right)\F_{\eta,k+1} \nonumber\\
 \leq & \ \F_{\eta,k}-2 \eta \gamma(\alpha-2)\left\langle z_{k+1} - z_{*} , \xi_{k+1}+C \left( z_{k+1} \right) \right\rangle\nonumber\\
& \ + \frac{1}{2}\left(1-\sqrt{\frac{9\alpha-2}{2(5\alpha-2)}}\right)\gamma^{2}\left(\nu_{3} k+\alpha\sqrt{\nu_{5} k+\alpha } +\nu_{4}\right) \left\lVert \xi_{k+1}+C \left( z_{k} \right) \right\rVert^2 \nonumber\\
& \ -\nu_{9} \left( k+1 \right) ^{2} \|C \left( z_{k+1} \right)-C \left( z_{k} \right)\|^2+ 2\left(1-\sqrt{\frac{9\alpha-2}{2(5\alpha-2)}}\right)\nu_{2} k \left\lVert z_{k+1} - z_{k} \right\rVert ^{2} \nonumber\\
& \ - \gamma^{2}\left(\nu_{5} k+\alpha\right)\left(\nu_{6} k+\nu_{7}\right)\|\xi_{k+1}+C \left( z_{k} \right) - (\xi_{k}+C \left( z_{k-1} \right))\|^2.\label{func:G1}
\end{align}
For $K_{3}(\eta):= \lceil\sqrt[3]{\nu_{8}^2}-1\rceil$ we see that for every $k \geq K_{3}(\eta)$
\begin{equation}
\left(1-\frac{\nu_{8}}{(k+1)\sqrt{k+1}}\right)^{-1}=\frac{(k+1)\sqrt{k+1}}{(k+1)\sqrt{k+1}-\nu_{8}}=1+\frac{\nu_{8}}{(k+1)\sqrt{k+1}-\nu_{8}}>1.\nonumber
\end{equation}
Since $\nu_{2},\nu_{3}<0$ and $\nu_{5},\nu_{6} > 0$, there exists an integer $K_{4} \geq 1$ such that for every $k\geq K_{4}$ it holds
\begin{equation}
\nu_{3} k+\alpha\sqrt{\nu_{5} k+\alpha } +\nu_{4}<0 \quad \textnormal{and} \quad \left(\nu_{5} k+\alpha\right)\left(\nu_{6} k+\nu_{7}\right)>0.\nonumber
\end{equation}
Hence, for every $k\geq K_{5}(\eta):=\max\{K_{2}(\eta), K_{3}(\eta), K_{4}\}$, according to (\ref{func:G1}), we obtain
\begin{align}
0 \leq \F_{\eta,k+1} \leq& \left( 1+\frac{\nu_{8}}{(k+1)\sqrt{k+1}-\nu_{8}}\right)\F_{\eta,k}-2 \eta \gamma(\alpha-2)\left\langle z_{k+1} - z_{*} , \xi_{k+1}+C \left( z_{k+1} \right) \right\rangle\nonumber\\
&+ \frac{1}{2}\left(1-\sqrt{\frac{9\alpha-2}{2(5\alpha-2)}}\right)\gamma^{2}\left(\nu_{3} k+\alpha\sqrt{\nu_{5} k+\alpha } +\nu_{4}\right) \left\lVert \xi_{k+1}+C \left( z_{k} \right) \right\rVert^2 \nonumber\\
&- \nu_{9} \left( k+1 \right) ^{2} \|C \left( z_{k+1} \right)-C \left( z_{k} \right)\|^2+ 2\left(1-\sqrt{\frac{9\alpha-2}{2(5\alpha-2)}}\right)\nu_{2} k \left\lVert z_{k+1} - z_{k} \right\rVert ^{2}\nonumber\\
&- \gamma^{2}\left(\nu_{5} k+\alpha\right)\left(\nu_{6} k+\nu_{7}\right)\|\xi_{k+1}+C \left( z_{k} \right) - (\xi_{k}+C \left( z_{k-1} \right))\|^2,\nonumber
\end{align}
we utilized the fact that all summands on the right-hand side of the inequality are nonpositive, except for the first one. The fact that $\left\langle z_{k+1} - z_{*} , \xi_{k+1}+C \left( z_{k+1} \right) \right\rangle \geq 0$ for every $k \geq 1$ follows from the monotonicity of the operator $M+C$ and relation \eqref{algo:inc-M}.
This means that for 
\begin{align*}
a_k:=&\,\F_{\eta,k}\geq 0,\nonumber\\
b_k:=&\,2 \eta \gamma(\alpha-2)\left\langle z_{k+1} - z_{*} ,\xi_{k+1}+C \left( z_{k+1} \right) \right\rangle-  2\left(1-\sqrt{\frac{9\alpha-2}{2(5\alpha-2)}}\right)\nu_{2} k \left\lVert z_{k+1} - z_{k} \right\rVert ^{2} \nonumber\\
&+\nu_{9} \left( k+1 \right) ^{2} \|C \left( z_{k+1} \right)-C \left( z_{k} \right)\|^2 + \gamma^{2}\left(\nu_{5} k+\alpha\right)\left(\nu_{6} k+\nu_{7}\right)\|\xi_{k+1}+C \left( z_{k} \right) - (\xi_{k}+C \left( z_{k-1} \right))\|^2\nonumber\\
&\,-\frac{1}{2}\left(1-\sqrt{\frac{9\alpha-2}{2(5\alpha-2)}}\right)\gamma^{2}\left(\nu_{3} k+\alpha\sqrt{\nu_{5} k+\alpha } +\nu_{4}\right) \left\lVert \xi_{k+1}+C \left( z_{k} \right) \right\rVert^2 \geq 0 \nonumber\\
d_k:=&\,\frac{\nu_{8}}{(k+1)\sqrt{k+1}-\nu_{8}}>0,\nonumber
\end{align*}
it holds
\begin{equation}\label{perturbeddec}
a_{k+1} \leq (1+d_k)a_k - b_k \quad \forall k \geq K_{5}(\eta),
\end{equation}
which, according \cref{ddec}, ensures that the relations (\ref{rate:sum}) hold as well as that the sequence $(\F_{\eta,k})_{k \geq 1}$ is convergent.

As $(\F_{\eta,k})_{k \geq 1}$ is also bounded from above,  (\ref{G:inequa}), (\ref{dec:mu}) and (\ref{def:S0}) yield that for every $k\geq K_{5}(\eta)$
\begin{align}
&\frac{3(\alpha-2)}{4(5\alpha-2)} \left\lVert 4 \eta \left( z_{k} - z_{*} \right)+ 2k \left( z_{k} - z_{k-1} \right) + \frac{5\alpha-2}{2(\alpha-1)}\gamma k \left(\xi_k +C \left( z_{k-1} \right)\right) \right\rVert ^{2} \nonumber\\
&	+\frac{3}{2}(\alpha-2)\left(\frac{1}{5\alpha-2}-\frac{1}{4(2-\varepsilon)(\alpha-1)}\right) k^2 \left\lVert z_k-z_{k-1} \right\rVert ^{2}\nonumber\\
&+  \eta(\alpha-1) \left( 1 - \frac{8\eta}{5\alpha-2} \right) \left\lVert z_{k} - z_{*} \right\rVert ^{2}\nonumber\\
\leq \ 	&\F_{\eta,k}\leq \sup_{k\geq 1}\F_{\eta,k}<+\infty.\nonumber
\end{align}
In particular, for every $k\geq K_{5}(\eta)$ it holds
\begin{align*}
	 &\left\lVert 4 \eta \left( z_{k} - z_{*} \right)+ 2k \left( z_{k} - z_{k-1} \right) + \frac{5\alpha-2}{2(\alpha-1)}\gamma k \left(\xi_k +C \left( z_{k-1} \right)\right) \right\rVert
	 \leq \nu_{10} := 2 \sqrt{\frac{5\alpha-2}{3(\alpha-2)}\sup_{k\geq 1}\F_{\eta,k} }<+\infty, \nonumber \\
	& k \left\lVert z_k-z_{k-1} \right\rVert 
	 \leq \nu_{11} := \sqrt{\frac{8(\alpha-1)(5\alpha-2)(2-\varepsilon)}{3(\alpha-2)\left(4(\alpha-1)(2-\varepsilon)-5\alpha+2\right)}\sup_{k\geq 1}\F_{\eta,k}} <+\infty, \nonumber \\
	& \left\lVert z_{k} - z_{*} \right\rVert 
	 \leq \nu_{12} := \sqrt{\frac{5\alpha-2}{\eta(\alpha-1)(5\alpha-2-8\eta)}\sup_{k\geq 1}\F_{\eta,k}} <+\infty. \nonumber
\end{align*}
Using the triangle inequality, we deduce from here that for every $k\geq K_{5}(\eta)$
\begin{align}
\left\lVert \xi_{k} + C \left( z_{k-1} \right) \right\rVert \leq & \ \frac{2(\alpha-1)}{(5\alpha-2) \gamma k}\left\lVert 4 \eta \left( z_{k} - z_{*} \right)+ 2k \left( z_{k} - z_{k-1} \right) + \frac{5\alpha-2}{2(\alpha-1)}\gamma k \left(\xi_k +C \left( z_{k-1} \right)\right) \right\rVert \nonumber\\
& \ +\frac{4(\alpha-1)}{(5\alpha-2) \gamma}\left\lVert z_k-z_{k-1} \right\rVert+\frac{8\eta(\alpha-1)}{(5\alpha-2)\gamma k}\left\lVert z_k-z_{*} \right\rVert\leq \frac{\nu_{13}}{k},\nonumber
\end{align}
where
\begin{equation}
\nu_{13}:=\frac{2(\alpha-1)}{(5\alpha-2) \gamma}(\nu_{10}+2\nu_{11}+4\overline{\eta}(\alpha)\nu_{12})>0.\nonumber
\end{equation}
Further, it follows from
\begin{equation}
\mysum_{k \geq 1} k^2 \left\lVert C \left( z_{k} \right) - C \left( z_{k-1} \right) \right\rVert ^{2} < + \infty\nonumber
\end{equation}
that
\begin{equation*}
\lim_{k\rightarrow +\infty}k\left\lVert C \left( z_{k} \right)-C \left( z_{k-1} \right) \right\rVert=0\quad \Rightarrow \quad \nu_{14}:=\sup_{k\geq 1}\left\{k\left\lVert C \left( z_{k} \right)-C \left( z_{k-1} \right) \right\rVert\right\}<+\infty.
\end{equation*}
This implies that for every $k\geq K_{5}(\eta)$
\begin{align*}
\|\xi_{k+1}+C \left( z_{k+1} \right)\|&\leq \|C \left( z_{k+1} \right)-C \left( z_{k} \right)\|+\|\xi_{k+1}+C \left( z_{k} \right)\| \leq \frac{\nu_{15}}{k},
\end{align*}
where $\nu_{15}:=\nu_{13}+\nu_{14}>0$. Since $(z_k)_{k\geq 0}$ is bounded, we deduce from here that for every $k\geq K_{5}(\eta)$
\begin{equation*}
0\leq \langle z_{k+1}-z_{*}, \xi_{k+1}+C \left( z_{k+1} \right)\rangle\leq\left\lVert z_{k+1}-z_{*} \right\rVert\|\xi_{k+1}+C \left( z_{k+1} \right)\|\leq \frac{\nu_{12}\nu_{15}}{k}.\nonumber
\end{equation*}
In conclusion,
\begin{align*}
&\left\lVert z_k-z_{k-1} \right\rVert = \bO \left( \dfrac{1}{k } \right),\,\quad  \left\lVert \xi_k+C \left( z_{k-1} \right) \right\rVert = \bO \left( \dfrac{1}{k} \right)\\
&\left\lVert \xi_k+C \left( z_{k} \right) \right\rVert = \bO \left( \dfrac{1}{k} \right),\,\quad
\left\langle z_{k} - z_{*} , \xi_k+C \left( z_{k} \right)\right\rangle = \bO \left( \dfrac{1}{k } \right) 
\end{align*}
as $k \rightarrow +\infty$.

In the remainder of the proof we will show that the sequence $(\E_{\eta,k})_{k\geq 1}$ is convergent and that it has the same limit as $(\F_{\eta,k})_{k\geq 1}$. On the one hand, the Cauchy-Schwarz inequality yields
\begin{align*}
0&\leq \lim_{k\rightarrow +\infty}k^2\big|\langle z_k-z_{k-1}, C \left( z_{k} \right)-C \left( z_{k-1} \right)\rangle\big|
\leq \nu_{11}\lim_{k\rightarrow +\infty}k\big\|C \left( z_{k} \right)-C \left( z_{k-1} \right)\big\|=0,
\end{align*}
therefore
\begin{align*}
0&\leq \lim_{k\rightarrow +\infty}k^2\langle z_k-z_{k-1}, C \left( z_{k} \right)-C \left( z_{k-1} \right)\rangle
=0.
\end{align*}
On the other hand,
\begin{align*}
0\leq\lim_{k\rightarrow +\infty} \sqrt{k}\left\lVert C \left( z_{k} \right)-C \left( z_{k-1} \right) \right\rVert^2&\leq\lim_{k\rightarrow +\infty} k\sqrt{k}\left\lVert C \left( z_{k} \right)-C \left( z_{k-1} \right) \right\rVert^2 \nonumber\\
&\leq \lim_{k\rightarrow +\infty} k^2\left\lVert C \left( z_{k} \right)-C \left( z_{k-1} \right) \right\rVert^2=0,\nonumber
\end{align*}
which, by using the definitions of the two discrete energy functions, allow us to conclude that
\begin{equation}
\lim_{k\rightarrow +\infty} \E_{\eta,k}=\lim_{k\rightarrow +\infty} \F_{\eta,k} \in \sR.\nonumber
\end{equation}	
\end{proof}

\begin{rmk}\label{rmk28}
The discrete energy function employed in the convergence analysis of the \emph{Fast Krasnosel'ski\u{\i}-Mann algorithm} (fKM), as introduced in \cite{Bot2023Nguyen} (also referenced in \cite{Bot2023Robert_Nguyen}), and forming the basis of the iterative scheme \eqref{fast_KM}, exhibits a nonnegative and nonincreasing nature.  In the current context, the introduction of additional structural elements inherent in the sum of two operators results in a discrete energy function that demonstrates a perturbed decreasing property, as observed in \eqref{perturbeddec}. However, this perturbation does not impede the convergence of the sequence $\left( \F_{\eta,k} \right)_{k \geq 1}$.
\end{rmk}

Next we prove the weak convergence of the generated sequence of the iterates and derive improved convergence rates.
\begin{thm}
\label{thm:ex:conv}
Let $z_* \in {\cal Z}$, $\left(z_{k} \right)_{k \geq 0}$ be the sequence generated by Algorithm \ref{algo:ffb}, and  $\left(\xi_{k} \right)_{k \geq 1}$ the sequence generated by the iterative scheme \eqref{algo:z-xi}. Then the following statements are true:
\begin{enumerate}
\item 
the sequence $\left(z_{k} \right) _{k \geq 0}$ converges weakly to a solution of \eqref{intro:pb:eq};

\item
it holds
\begin{align*}
\left\lVert z_k-z_{k-1} \right\rVert = o \left( \dfrac{1}{k } \right),\,\quad  \left\langle z_{k} - z_{*} , \xi_k+C \left( z_{k} \right)\right\rangle = o \left( \dfrac{1}{k } \right),
\,\quad \left\lVert\xi_k+C \left( z_{k} \right) \right\rVert = o \left( \dfrac{1}{k} \right) \quad \mbox{as} \ k \rightarrow + \infty.
\end{align*}
Since, as seen in \cref{rmk:sub}, $\xi_k\in M(z_k)$ for every $k\geq 1$, it yields
\begin{equation*}
\dist (0, M(z_k) + C(z_k)) = o \left( \dfrac{1}{k } \right) \quad \mbox{and} \quad \left\lVert z_k-J_{\gamma M}\big(z_k-\gamma C \left( z_{k} \right)\big)\right\rVert  = o \left( \dfrac{1}{k } \right) \quad \mbox{as} \ k \rightarrow + \infty.
\end{equation*}
\end{enumerate}
\end{thm}

\begin{proof}
We fix $0<\varepsilon<\frac{3(\alpha-2)}{4(\alpha-1)}<1$ such that $0 < \gamma< \frac{2}{2-\varepsilon} \beta$,  consider the parameters $\underline{\eta}(\alpha)$ and $\overline{\eta}(\alpha)$ provided by Lemma \ref{lemm2:3} (ii) and choose $\eta \in \left( \underline{\eta}(\alpha), \overline{\eta}(\alpha) \right)$. 

(i) For every $k \geq 1$ we set
\begin{align}
p_{k}	& := \dfrac{1}{2} \left( \alpha - 1 \right) \left\lVert z_{k} - z_{*} \right\rVert ^{2} + k \left\langle z_{k} - z_{*} , z_{k} - z_{k-1} + \gamma  (\xi_k+C \left( z_{k-1} \right)) \right\rangle , \label{defi:im:p-k} \\
q_{k}	& := \dfrac{1}{2} \left\lVert z_{k} - z_{*} \right\rVert ^{2} + \gamma \mysum_{i = 1}^{k} \left\langle z_{i} - z_{*} , \xi_i+C(z_{i-1}) \right\rangle, \label{defi:im:q-k}
\end{align}
and notice that
\begin{align*}
q_{k} - q_{k-1}
& = \left\langle z_{k} - z_{*} , z_{k} - z_{k-1} \right\rangle - \dfrac{1}{2} \left\lVert z_{k} - z_{k-1} \right\rVert ^{2} + \gamma \left\langle z_{k} - z_{*} , \xi_k+C \left( z_{k-1} \right) \right\rangle ,
\end{align*}
and thus
\begin{equation*}
\left( \alpha - 1 \right) q_{k} + k \left( q_{k} - q_{k-1} \right) = p_{k} + \left( \alpha - 1 \right) \gamma \mysum_{i = 1}^{k} \left\langle z_{i} - z_{*} , \xi_i+C(z_{i-1}) \right\rangle - \dfrac{k}{2} \left\lVert z_{k} - z_{k-1} \right\rVert ^{2} .
\end{equation*}

Invoking \eqref{im:defi:E-k:eq}, we can rewrite the discrete energy function for every $k \geq 1$ as
\begin{align}\label{represE}
\E_{\eta,k} = & \ 4 \eta p_{k}  + 2 \eta \gamma \alpha\langle z_k-z_*,\xi_k+C \left( z_{k-1} \right)\rangle+\dfrac{1}{2} k^{2} \left\lVert 2 \left( z_{k} - z_{k-1} \right) + \frac{5\alpha-2}{4(\alpha-1)}\gamma (\xi_k+C \left( z_{k-1} \right)) \right\rVert ^{2} \nonumber \\
&+ \frac{1}{2} \gamma^{2}\left(\frac{3(\alpha-2)}{4(\alpha-1)}k +\alpha\right) \left( \frac{5\alpha-2}{4(\alpha-1)} k + \alpha \right) \left\lVert \xi_k +C \left( z_{k-1} \right) \right\rVert ^{2}.
\end{align}
From here, we easily see that for every $0\leq\underline{\eta}(\alpha) < \eta_{1} < \eta_{2} < \overline{\eta}(\alpha)$ and  every $k \geq 1$ it holds
\begin{align}
\label{conv:im:Ed-lambad}
\E_{\eta_{2},k} - \E_{\eta_{1},k}
= &4 \left( \eta_{2} - \eta_{1} \right) \left( p_{k} +\frac{1}{2}\alpha \gamma\langle z_k-z_*,\xi_k+C \left( z_{k-1} \right)\rangle \right).
\end{align}
Since the limit $\lim_{k \to + \infty} \left(\E_{\eta_{2},k} - \E_{\eta_{1},k}\right) \in \sR$ exists and 
\begin{align}\label{deczk}
\lim_{k\rightarrow +\infty}\langle z_k-z_*, \xi_k+C \left( z_{k-1} \right)\rangle
= \lim_{k\rightarrow +\infty}\left\lVert \xi_k+C \left( z_{k-1} \right) \right\rVert = 0,
\end{align}
according to \eqref{conv:im:Ed-lambad}
\begin{equation}
\label{conv:lim-p-k}
\lim\limits_{k \to + \infty} p_{k} \in \sR \textrm{ exists}.
\end{equation}	

Moreover, since for every $k \geq 1$
\begin{align}
\mysum_{i = 1}^{k}\big|\langle z_i-z_*,C(z_{i-1})-C(z_{i})\rangle\big|&\leq \mysum_{i = 1}^{k}\left\lVert z_i-z_* \right\rVert\left\lVert C(z_{i-1})-C(z_{i}) \right\rVert \nonumber\\
&\leq \frac{1}{2}\mysum_{i = 1}^{k} \frac{1}{i^2}\left\lVert z_i-z_* \right\rVert^2+\frac{1}{2}\mysum_{i = 1}^{k} i^2\left\lVert C(z_{i})-C(z_{i-1}) \right\rVert^2 \nonumber\\
&\leq \frac{1}{2}\nu_{12}^2\mysum_{i = 1}^{k} \frac{1}{i^2}+\frac{1}{2}\mysum_{i = 1}^{k} i^2\left\lVert C(z_{i})-C(z_{i-1}) \right\rVert^2 \nonumber\\
&\leq \frac{1}{2}\nu_{12}^2\mysum_{i = 1}^{+\infty} \frac{1}{i^2}+\frac{1}{2}\mysum_{i = 1}^{+\infty} i^2\left\lVert C(z_{i})-C(z_{i-1}) \right\rVert^2<+\infty, \nonumber
\end{align}
 the series $\sum_{k\geq 2}\langle z_k-z_*,C \left( z_{k-1} \right)-C \left( z_{k} \right)\rangle$ is absolutely convergent, thus convergent. It follows from here that the limit
\begin{align*}
\lim_{k\rightarrow +\infty}\mysum_{i=1}^{k}\langle z_i-z_*,\xi_i+C(z_{i-1})\rangle
=\!\!\lim_{k\rightarrow +\infty}\mysum_{i=1}^{k}\langle z_i-z_*,\xi_i+C(z_{i})\rangle+\!\!\lim_{k\rightarrow +\infty}\!\!\mysum_{i=1}^{k}\langle z_i-z_*,C(z_{i-1})-C(z_{i})\rangle\in\sR \nonumber
\end{align*}
exists. Since, in addition,
\begin{equation*}
\lim\limits_{k \to + \infty} k \left\lVert z_{k+1} - z_{k} \right\rVert ^{2} = 0,
\end{equation*}
we obtain
\begin{equation*}
\lim\limits_{k \to + \infty}  \left( \alpha - 1 \right) q_{k} + k \left( q_{k} - q_{k-1} \right) \in \sR \textrm{ exists} .
\end{equation*}
We also have that $\left( q_{k} \right)_{k \geq 1}$ is bounded due to the boundedness of $\left( z_{k} \right)_{k \geq 0}$ and the fact that $$\lim_{k\rightarrow +\infty}\sum_{i=1}^k \left\langle z_{i} - z_{*} , \xi_{i}+C(z_{i-1}) \right\rangle \in\sR ,$$ which, according to \cref{lem:lim-u-k}, allows us to conclude that $\lim_{k \to + \infty} q_{k} \in \sR$ also exists. Once again,  by the definition of $q_{k}$ and the fact that the sequence $\left(\sum_{i = 1}^{k} \left\langle z_{i} - z_{*} , \xi_{i}+C(z_{i-1}) \right\rangle \right)_{k \geq 1}$ converges, it follows that $\lim_{k \to + \infty} \left\lVert z_{k} - z_{*} \right\rVert \in \sR$ exists. In other words, since $z_* \in {\cal Z}$ was arbitrarily chosen, the hypothesis \ref{lem:Opial:dis:i} in Opial’s Lemma (see \cref{lem:Opial:dis}) is fulfilled.

Now let $\widebar{z}$ be a weak sequential cluster point of $\left(z_{k} \right) _{k \geq 0}$, meaning that there exists a subsequence $\left(z_{k_{n}} \right)_{n \geq 0}$ such that
\begin{equation*}
z_{k_{n}} \rightharpoonup \widebar{z} \textrm{ as } n \to + \infty .
\end{equation*}
From \cref{prop:lim} we have
\begin{equation*}
\xi_{k_n}+C(z_{k_n}) \to 0 \textrm{ as } n \to + \infty , \textrm{where } \xi_{k_n} \in M \left(z_{k_{n}} \right)\quad \forall n \geq 0 .
\end{equation*}
Since the sum $M+C$ is maximally monotone,  it is sequentially closed in $\sH^{\textrm{weak}} \times \sH^{\textrm{strong}}$. This gives that $0\in M(\overline{z})+C(\overline{z})$, thus $\widebar{z} \in \cal Z$, and shows that hypothesis \ref{lem:Opial:dis:ii} of Opial’s Lemma is also fulfilled. This completes the proof of the weak convergence of the sequence $(z_k)_{k \geq 0}$ to a zero of the operator $M+C$.

(ii) According to \cref{prop:lim} we have that
\begin{equation*}
\lim\limits_{k \to + \infty} k  \left\lVert \xi_k +C \left( z_{k-1} \right)  \right\rVert ^{2} = 0 \quad \textnormal{and} \quad \lim\limits_{k \to + \infty}  \left\lVert \xi_k +C \left( z_{k-1} \right)  \right\rVert ^{2} = 0,
\end{equation*}
and the sequence $(\E_{\eta,k})_{k \geq 1}$ is convergent.  Taking into account \eqref{represE}, by using (\ref{deczk}) and the fact that the limit $\lim_{k \to + \infty} p_{k} \in \sR$ exists, it follows that for the sequence
$$g_{k}:= \dfrac{k^{2}}{2} \left( \left\lVert 2 \left( z_{k} - z_{k-1} \right) + \frac{5\alpha-2}{4(\alpha-1)}\gamma (\xi_k+C \left( z_{k-1} \right)) \right\rVert ^{2} + \frac{3(\alpha-2)(5\alpha-2)}{16(\alpha-1)^2} \gamma^{2} \left\lVert \xi_k +C \left( z_{k-1} \right) \right\rVert ^{2} \right),$$
the limit
\begin{equation*}
\lim\limits_{k \to + \infty}  g_{k} \in [0,+\infty) \textrm{ exists}.
\end{equation*}
Furthermore, \cref{prop:lim}  guarantees that
\begin{align*}
\mysum_{k \geq 1} \dfrac{1}{k} g_{k}
& \leq 4 \mysum_{k \geq 1} k \left\lVert z_{k} - z_{k-1} \right\rVert ^{2} +\frac{(13\alpha-10)(5\alpha-2)}{32(\alpha-1)^2} \gamma^{2} \mysum_{k \geq 1} k \left\lVert \xi_k+C \left( z_{k-1} \right) \right\rVert ^{2}  < + \infty .
\end{align*}
From here we conclude that $\lim_{k \to + \infty} \frac{1}{k} g_{k} = 0$, and so
\begin{equation*}
\lim\limits_{k \to + \infty} k  \left\lVert 2 \left( z_{k} - z_{k-1} \right) + \frac{5\alpha-2}{4(\alpha-1)}\gamma (\xi_k+C \left( z_{k-1} \right)) \right\rVert = \lim\limits_{k \to + \infty} k \left\lVert \xi_k +C \left( z_{k-1} \right) \right\rVert  = 0 .
\end{equation*}
This immediately implies $\lim_{k \to + \infty} k \left\lVert z_k-z_{k-1} \right\rVert  = 0$ and, by using the Lipschitz continuity of $C$, that $\lim\limits_{k \to + \infty} k \left\lVert \xi_k+C \left( z_{k} \right) \right\rVert = 0$.
Further, using the Cauchy-Schwarz inequality and the fact that $(z_k)_{k\geq 0}$ is bounded, we obtain that
\begin{equation*}
\lim\limits_{k \to + \infty} k \left\langle z_k-z_{*}, \xi_k+C \left( z_{k} \right) \right\rangle = 0.
\end{equation*}
Finally, it follows from \cref{bb1} that 
\begin{align*}
0 \leq \lim\limits_{k \to + \infty} k  \|z_k-J_{\gamma M}\big(z_k-\gamma C \left( z_{k} \right)\big)\| \leq \gamma \lim\limits_{k \to + \infty} k \dist(0, M(z_{k}) + C ( z_{k} ) ) \leq \gamma \lim\limits_{k \to + \infty} k \left\lVert \xi_k+C \left( z_{k} \right) \right\rVert = 0, 
\end{align*}
which completes the proof.
\qedhere
\end{proof}

\section{A fast primal-dual full splitting algorithm}
\label{sec:convex}

The aim of this section is to use the Fast Forward-Backward algorithm to construct a primal-dual full splitting algorithm for nonsmooth convex optimization problems with linear constraints that exhibits fast convergence rates as well as convergence of the iterates. 

The problem under consideration reads
\begin{align}
\label{convexpro}
\begin{array}{rl}
\min\limits_{x\in \sX} & f \left( x \right)+h \left( x \right), \\
\textrm{subject to} 	& Ax = b,
\end{array}
\end{align}
where $\sX$ and $\sY$ are real Hilbert spaces, $f \colon \sX \rightarrow \mathbb{R} \cup \left\lbrace + \infty \right\rbrace$ is a proper, lower semicontinuous, and convex function, $h \colon \sX \rightarrow \mathbb{R}$ convex differentiable function such that $\nabla h$ is $\frac{1}{\beta}${-}Lipschitz continuous  with $\beta > 0$, and $A \colon \sX\rightarrow \sY$ is a  linear continuous operator.
Consider the saddle point problem associated to problem \eqref{convexpro}
\begin{equation*}
\min_{x \in \sX} \max_{\lambda \in \sY} \Lag \left( x , \lambda \right),
\end{equation*}
where $\Lag \colon \sX \times \sY \to \sR$ denotes the associated \emph{Lagrangian function}
\begin{equation*}
\Lag \left( x , \lambda \right) := f \left( x \right)+h \left( x \right) + \left\langle \lambda , Ax - b \right\rangle .
\end{equation*}
A pair $\left( x_{*} , \lambda_{*} \right) \in \sX \times \sY$ is said to be a \emph{saddle point} of the Lagrangian function $\Lag$ if for every $\left( x , \lambda \right) \in \sX \times \sY$
\begin{equation*}
\Lag \left( x_{*} , \lambda \right) \leq \Lag \left( x_{*} , \lambda_{*} \right) \leq \Lag \left( x , \lambda_{*} \right).
\end{equation*}
If $\left( x_{*} , \lambda_{*} \right) \in \sX \times \sY$ is a saddle point of $\Lag$, then $x_{*} \in \sX$ is an optimal solution of \eqref{convexpro} and $\lambda_{*} \in \sY$ is an optimal solution of its Lagrange dual problem. If $x_{*} \in \sX$ is an optimal solution of \eqref{convexpro} and a suitable constraint qualification is fulfilled (see, for instance, \cite{Bot2010,bauschkebook2011}), then there exists an optimal solution $\lambda_{*} \in \sY$ of the Lagrange dual problem of \eqref{convexpro} such that $\left( x_{*} , \lambda_{*} \right) \in \sX \times \sY$ is a saddle point of $\Lag$. The set of saddle points of $\Lag$, called the \emph{set of primal-dual optimal solutions} of \eqref{convexpro}, will be denoted by $\sol$.  

The system of primal-dual optimality conditions for \eqref{convexpro} reads
\begin{equation}
\label{intro:opt-Lag}
\left( x_{*} , \lambda_{*} \right) \in \sol
\Leftrightarrow \begin{cases}
0 \in \partial_{x} \Lag \left( x_{*} , \lambda_{*} \right)  \\
0 = \nabla_{\lambda} \Lag \left( x_{*} , \lambda_{*} \right) 
\end{cases} \Leftrightarrow \begin{cases}
0 \in  \partial f \left( x_{*} \right) +\nabla h \left( x_{*} \right)+ A^{*} \lambda_{*} 	\\
0 = Ax_{*} - b 												
\end{cases},
\end{equation}
where $A^{*} : \sY \rightarrow \sX$ denotes the adjoint operator of $A$.

For $\bH = \sX \bigoplus \sY$, the problem of finding a primal-dual solution of \eqref{convexpro} can be equivalently written as the monotone inclusion
\begin{equation}\label{no100}
\textrm{find} \quad \bz\in \bH \quad \textrm{such} \quad \textrm{that} \quad \boldsymbol{0}\in \bP(\bz)+\bQ(\bz),
\end{equation}
where
\begin{equation}\
\left\{
\begin{aligned}\label{lef15}
& \bP \colon \bH \rightarrow 2^{\bH} \colon \left( x , \lambda \right) \mapsto \big(\partial f(x)+A^{*} \lambda\big)\times\big(-Ax +b\big), \\
& \bQ \colon \bH \rightarrow \bH \colon \left( x , \lambda \right) \mapsto (\nabla h(x),0).
\end{aligned}\right.
\end{equation}
The operator $\bP$ as the sum of the maximally monotone operators $ \left( x , \lambda \right) \mapsto \partial f(x)\times \left\lbrace b \right\rbrace$ and $\left( x , \lambda \right) \mapsto (A^{*} \lambda,-Ax)$ is also a maximally monotone operator. In view of the Baillon-Haddad theorem, $\nabla h$ is $\beta$-cocoercive, which means that $\bQ$ is also $\beta$-cocoercive.

Inspired by \cite{Vu2011}, we define the self-adjoint operator
\begin{align}\label{def:N}
& \qquad \bN \colon \bH \rightarrow \bH, \quad (x,\lambda)\mapsto \left( \frac{1}{\tau} x-A^{*} \lambda,-A x+\frac{1}{\sigma} \lambda \right),
\end{align}
which is $\rho$-strongly positive with modulus 
$$\rho:=\min \left\lbrace \frac{1}{\tau},\frac{1}{\sigma} \right\rbrace \left(1-\sqrt{\tau\sigma\|A\|^2}\right).$$ Therefore,
its inverse $\bN^{-1}$ exists and $\|\bN^{-1}\|\leq\frac{1}{\rho}$. This means that $\bN^{-1}\bP$ is a maximally monotone operator and $\bN^{-1}\bQ$ is a $\beta\rho$-cocoercive operator.  The algorithm below is the Fast Forward-Backward algorithm applied to the problem of finding a zero of $\bN^{-1}\bP + \bN^{-1}\bQ$.

\begin{mdframed}
	\begin{algo}
		\label{algo:fpdVu}
		Let $\alpha >2$, $\tau,\sigma>0$ such that $$1<\min \left\lbrace \frac{1}{\tau},\frac{1}{\sigma} \right\rbrace \left(1-\sqrt{\tau\sigma\|A\|^2}\right) \frac{8(\alpha-1)}{(5\alpha-2)}\beta,$$
$x_{0}, v_{0}\in \sX $, $\lambda_{0}, \eta_{0} \in \sY$, and $x_1 := \textnormal{prox}_{\tau f}\left(v_{0}-\tau A^*\eta_0-\tau \nabla h\left( x_{0} \right) \right)$ and $\lambda_{1} := \eta_0 + \sigma \left( Ax_{1} - b \right) + \sigma A \left( x_{1} - v_{0} \right)$. We set
		\begin{equation}\label{Main_algo}
			(\forall k \geq 1) \quad \begin{dcases}
				v_{k}		& := x_{k}+\left(1-\frac{\alpha}{k+\alpha}\right)\left( x_{k}-x_{k-1} \right)+\left(1-\frac{\alpha}{2(k+\alpha)}\right) \left( v_{k-1} - x_{k} \right) ;\\
				\eta_{k}	& := \lambda_k+\left(1-\frac{\alpha}{k+\alpha}\right)(\lambda_k-\lambda_{k-1})+\left(1-\frac{\alpha}{2(k+\alpha)}\right) \left( \eta_{k-1} - \lambda_{k} \right) ;\\
				x_{k+1}		& := \textnormal{prox}_{\tau f}\left(v_{k}-\tau A^*\eta_k-\tau \nabla h\left( x_{k} \right) \right);\\
				\lambda_{k+1}	& := \eta_k + \sigma \left( Ax_{k+1} - b \right) + \sigma A \left( x_{k+1} - v_{k} \right).\nonumber
			\end{dcases}
		\end{equation}
	\end{algo}
\end{mdframed}
Based on \cref{thm:ex:conv}, we obtain the following convergence and convergence rate results.

\begin{thm}\label{thm:comp:conv}
	Let $(x_{*},\lambda_{*}) \in {\cal S}$, and $\left(x_{k},\lambda_{k} \right)_{k \geq 0}$ be the sequence generated by Algorithm \ref{algo:fpdVu}. Then the following statements are true:
	\begin{enumerate}
		\item 
		the sequence $\left(x_{k} , \lambda_{k} \right) _{k \geq 0}$ converges weakly to a primal-dual solution of \eqref{convexpro};
		
		\item 
		we have
		\begin{align*}
			\left\lVert x_{k} - x_{k-1} \right\rVert = o \left( \dfrac{1}{k} \right)	
			& \qquad \textrm{ and } \qquad
			\left\lVert \lambda_{k} - \lambda_{k-1} \right\rVert = o \left( \dfrac{1}{k} \right) \nonumber \\
			\left\lVert w_k + \nabla h\left( x_{k} \right) \right\rVert = o \left( \dfrac{1}{k} \right)
			& \qquad \textrm{ and } \qquad
			\left\lVert Ax_{k} - b \right\rVert = o \left( \dfrac{1}{k} \right) \nonumber \\
			\Lag \left( x_{k} , \lambda_{*} \right)-\Lag \left( x_{*} , \lambda_{k} \right) = o \left( \dfrac{1}{k } \right)
			& \qquad \textrm{ and } \qquad
			\left\lvert \left( f+h \right) \left( x_{k} \right) - \left( f+h \right) \left( x_{*} \right) \right\rvert = o \left( \dfrac{1}{k} \right)
		\end{align*}
as $k \to + \infty$, where, for every $k \geq 1$
		\begin{equation}
			w_{k+1} := \frac{1}{\tau} \left( v_{k}-x_{k+1} \right) +A^{*} \left( \lambda_{k+1} - \eta_{k} \right) -\nabla h\left( x_{k} \right) \in \partial f(x_{k+1}) +A^*\lambda_{k+1}.\nonumber
		\end{equation}
	\end{enumerate}
\end{thm}

\begin{proof}
We apply \cref{algo:ffb} with step size $\gamma=1$ to the problem
\begin{equation}\label{renormedincl}
\textrm{find} \quad \bz\in \bH \quad \textrm{such} \quad \textrm{that} \quad \boldsymbol{0}\in \bN^{-1}\bP(\bz)+ \bN^{-1}\bQ(\bz).
\end{equation}
For the initial points $\boldsymbol{z}_{0}, \boldsymbol{y}_{0} \in \bH$ and  $\boldsymbol{z}_{1}:= J_{\boldsymbol{N}^{-1}\boldsymbol{P}} \Big( \boldsymbol{y}_{0} -\boldsymbol{N}^{-1}\boldsymbol{Q}(\boldsymbol{z}_{0})\Big)$, this gives rise to the following iterative scheme
\begin{equation}
(\forall k \geq 1) \quad  \left\{
\begin{aligned}\label{2222}
\boldsymbol{y}_{k}:= & \boldsymbol{z}_{k} + \left( 1 - \dfrac{\alpha}{k + \alpha} \right) \left( \boldsymbol{z}_{k} - \boldsymbol{z}_{k-1} \right) + \left(1-\frac{\alpha}{2(k+\alpha)}\right) \left(\boldsymbol{y}_{k-1} - \boldsymbol{z}_{k}\right);\\
\boldsymbol{z}_{k+1}:= & J_{\boldsymbol{N}^{-1}\boldsymbol{P}} \Big( \boldsymbol{y}_{k} -\boldsymbol{N}^{-1}\boldsymbol{Q}(\boldsymbol{z}_{k})\Big). \\
\end{aligned}\right.
\end{equation}
The second equation (\ref{2222}) can be equivalently written  for every $k \geq 1$ as
\begin{equation*}
\boldsymbol{z}_{k+1}+\boldsymbol{N}^{-1}\boldsymbol{P}(\boldsymbol{z}_{k+1})\ni  \boldsymbol{y}_{k} -\boldsymbol{N}^{-1}\boldsymbol{Q}(\boldsymbol{z}_{k})
\end{equation*}
and further as
\begin{equation}\label{cf22}
\boldsymbol{N} \boldsymbol{z}_{k+1}+\boldsymbol{P}(\boldsymbol{z}_{k+1})\ni  \boldsymbol{N} \boldsymbol{y}_k-\boldsymbol{Q}(\boldsymbol{z}_k).
\end{equation}
By denoting $\boldsymbol{z}_k:=(x_{k},\lambda_{k})$ and $\boldsymbol{y}_k:=(v_{k},\eta_{k})$ for every $k \geq 0$, one can see that the first equation in (\ref{2222}) corresponds to the first two equations in the iterative scheme of Algorithm \ref{algo:fpdVu}. On the other hand, by making use of the definitions of $\boldsymbol{N}$, $\boldsymbol{P}$ and $\boldsymbol{Q}$, \eqref{cf22} is nothing else than 
\begin{equation*}\
\left\{
\begin{aligned}
&\frac{1}{\tau}x_{k+1}+\partial f(x_{k+1})\ni\frac{1}{\tau} v_{k}- A^*\eta_{k}- \nabla h\left( x_{k} \right),\\
&\frac{1}{\sigma}\lambda_{k+1}+b-2Ax_{k+1}=- Av_{k} +\frac{1}{\sigma}\eta_{k}
\end{aligned}\right.
\end{equation*}
or, equivalently,
\begin{equation*}
\left\{
\begin{aligned}
&x_{k+1}=\textnormal{prox}_{\tau f}\left(v_{k}-\tau A^*\eta_k
-\tau \nabla h\left( x_{k} \right) \right),\\
&\lambda_{k+1}	 = \eta_k + \sigma \left( Ax_{k+1} - b \right) + \sigma A \left( x_{k+1} - v_{k} \right),
\end{aligned}\right.
\end{equation*}
which corresponds to the last equations in the iterative scheme of Algorithm \ref{algo:fpdVu}. The initial condition $\boldsymbol{\xi}_1\in \boldsymbol{N}^{-1}\boldsymbol{P}(\boldsymbol{z}_{1})$ is also equivalent the initialization made for $x_1$ and $\lambda_1$. This proves that Algorithm \ref{algo:fpdVu} is nothing else than Algorithm \ref{algo:z-y} applied to a particular monotone inclusion problem, which means that it inherits its convergence properties.

For the formulation of the convergence rates, the sequence $(\boldsymbol{\xi}_{k})_{k \geq 1}$ introduced in Propostion \ref{rmk:sub} plays a crucial role. In the context of the iterative scheme \eqref{2222}, this is defined for every $k \geq 1$ as
$$ \boldsymbol{\xi}_{k}:=\boldsymbol{y}_{k-1}  - \boldsymbol{N}^{-1}\boldsymbol{Q}(\boldsymbol{z}_{k-1})-\boldsymbol{z}_{k}$$
or, equivalently,
\begin{align} \label{cf23}
\boldsymbol{N}\boldsymbol{\xi}_{k}=\boldsymbol{N} \boldsymbol{y}_{k-1}-\boldsymbol{Q}(\boldsymbol{z}_{k-1})-\boldsymbol{N} \boldsymbol{z}_{k}.
\end{align}
By denoting $\boldsymbol{N}\boldsymbol{\xi}_k:= (w_k,\zeta_{k})$,  the formula \eqref{cf23} can be equivalently written as
\begin{equation}\
\left\{
\begin{aligned}\label{inequ2001}
w_{k}&=\frac{1}{\tau}v_{k-1}-A^*\eta_{k-1} - \nabla h\left( x_{k-1} \right)-\frac{1}{\tau}x_{k}+A^*\lambda_{k},\\
\zeta_{k}&= -Av_{k-1}+\frac{1}{\sigma}\eta_{k-1}+Ax_{k} -\frac{1}{\sigma}\lambda_{k}.
\end{aligned}\right.
\end{equation}
Furthermore, according to \eqref{algo:inc-M}, for every $k \geq 1$ it holds $\boldsymbol{N}\boldsymbol{\xi}_{k} \in \boldsymbol{P} \boldsymbol{z}_{k}$ or, equivalently,
\begin{equation}\
\left\{
\begin{aligned}\label{w in sf}
w_{k}&\in \partial f(x_{k})+A^*\lambda_{k},\\
\zeta_{k}&= -Ax_{k}+b.
\end{aligned}\right.
\end{equation}

(i) \cref{thm:ex:conv} (i) guarantees that the sequence $\left( \bz_{k} \right) _{k \geq 0}$  generated by the iterative scheme \eqref{2222} converges weakly to an element of $\textnormal{zer} \left( \bN^{-1}\bP + \bN^{-1}\bQ \right) = \textnormal{zer} \left( \bP + \bQ \right)$. In other words, the sequence $\left( x_{k} , \lambda_{k} \right) _{k \geq 0}$ generated by Algorithm \ref{algo:fpdVu} converges weakly to a primal-dual optimal solution of \eqref{convexpro}.

(ii) \cref{thm:ex:conv} (ii) yields
\begin{align*}
\left\lVert \bz_k-\bz_{k-1} \right\rVert = o \left( \dfrac{1}{k } \right),\,\quad  
\left\langle \bz_{k} - \bz_{*} , \bxi_k+\bN^{-1}\bQ(\bz_{k})\right\rangle = o \left( \dfrac{1}{k } \right),\,\quad  
\left\lVert \bxi_k+\bN^{-1}\bQ(\bz_{k}) \right\rVert = o \left( \dfrac{1}{k} \right)
\end{align*}
as $k \rightarrow +\infty$.
Since
\begin{equation*}
\left\lVert x_{k}-x_{k-1} \right\rVert + \left\lVert \lambda_{k}-\lambda_{k-1} \right\rVert \leq \sqrt{2} \sqrt{\left\lVert x_{k}-x_{k-1} \right\rVert ^{2} + \left\lVert \lambda_{k}-\lambda_{k-1} \right\rVert ^{2}} = \sqrt{2} \left\lVert \bz_k-\bz_{k-1} \right\rVert \ \forall k \geq 1,
\end{equation*}
we obtain
\begin{align*}
\left\lVert x_k - x_{k-1} \right\rVert = o \left( \dfrac{1}{k } \right) \quad \mbox{and} \quad \left\lVert \lambda_k-\lambda_{k-1} \right\rVert = o \left( \dfrac{1}{k } \right) \ \mbox{as} \ k \rightarrow +\infty.
\end{align*}
Since
\begin{align*}
0&\leq \lim_{k\rightarrow +\infty} k\left\lVert \bN\bxi_k+\bQ(\bz_{k}) \right\rVert
\leq \|\bN\|\lim_{k\rightarrow +\infty}k\left\lVert \bxi_k+\bN^{-1}\bQ(\bz_{k}) \right\rVert=0,
\end{align*}
it yields
\begin{align*}
& \left\lVert \bN\bxi_k+\bQ(\bz_{k}) \right\rVert = o \left( \dfrac{1}{k} \right) \ \mbox{as} \ k \rightarrow +\infty,
\end{align*}
and further, by the boundedness of $\left( \bz_{k} \right) _{k \geq 0}$,
\begin{align*}
&\left\langle \bz_{k} - \bz_{*} , \bN\bxi_k+\bQ(\bz_{k})\right\rangle = o \left( \dfrac{1}{k} \right) \ \mbox{as} \ k \rightarrow +\infty.
\end{align*}
For $\bz_k=(x_{k},\lambda_{k})$ and  $\bN\bxi_k= (w_k,b-Ax_{k})$, where $w_k =\frac{1}{\tau}v_{k-1}-A^* \eta_{k-1}-\nabla h\left( x_{k-1} \right)-\frac{1}{\tau}x_{k}+A^*\lambda_{k}$, this yields 
\begin{align}\label{defiopt}
& \left\lVert \bN\bxi_k+\bQ(\bz_{k}) \right\rVert = \left\lVert \begin{pmatrix} w_k + \nabla h\left( x_{k} \right) , b-Ax_{k} \end{pmatrix} \right\rVert = o \left( \dfrac{1}{k} \right) \ \mbox{as} \ k \rightarrow +\infty
\end{align}
and
\begin{align}\label{Inrate:1}
&\left\langle \begin{pmatrix} x_{k} -  x_{*} , \lambda_{k} - \lambda_{*} \end{pmatrix} , \begin{pmatrix} w_k + \nabla h\left( x_{k} \right) , b-Ax_{k} \end{pmatrix} \right\rangle = o \left( \dfrac{1}{k } \right) \ \mbox{as} \ k \rightarrow +\infty,
\end{align}
respectively. By arguing as above, we obtain
\begin{equation}\label{rateAb}
\left\lVert w_k + \nabla h\left( x_{k} \right) \right\rVert = o \left( \dfrac{1}{k } \right) \quad \mbox{and} \quad \left\lVert Ax_{k} - b \right\rVert  = o \left( \dfrac{1}{k} \right) \ \mbox{as} \ k \rightarrow +\infty.
\end{equation}
Recall that from \eqref{w in sf} we have $w_{k}\in \partial f\left( x_{k} \right)+A^* \lambda_{k}$.  Using the convexity of $f$ and $h$,  for every $k \geq 1$ it yields
\begin{align*}
&\langle w_k - A^*\lambda_k,x_{k}-x_{*}\rangle \geq f\left( x_{k} \right)-f(x_{*}), \\
&\langle \nabla h\left( x_{k} \right),x_{k}-x_{*}\rangle \geq h\left( x_{k} \right)-h(x_{*}),
\end{align*}
and, by adding the two inequalities,
\begin{align}
& \left\langle \bz_{k} - \bz_{*} , \bN\bxi_k+\bQ(\bz_{k})\right\rangle \nonumber \\
= \ 	& \langle w_k + \nabla h\left( x_{k} \right),x_{k}-x_{*}\rangle+\langle b-Ax_{k},\lambda_{k}-\lambda_{*}\rangle \nonumber \\
= \ 	& \langle w_k - A^*\lambda_k, x_{k}-x_{*}\rangle+\langle \nabla h\left( x_{k} \right),x_{k}-x_{*}\rangle+\langle b-Ax_{k},\lambda_{k}-\lambda_{*}\rangle+\langle A^*\lambda_k, x_{k}-x_{*} \rangle \nonumber\\
\geq \ 	& f\left( x_{k} \right)-f(x_{*}) + h\left( x_{k} \right)-h(x_{*}) + \langle \lambda_{*} , Ax_{k}-b\rangle \label{Inq:pre func} \nonumber\\
= \ 	& \Lag \left( x_{k} , \lambda_{*} \right)-\Lag \left( x_{*} , \lambda_{k} \right) \geq 0.
\end{align}
Combining this esitmate with (\ref{Inrate:1}) yields
\begin{align*}
&\Lag \left( x_{k} , \lambda_{*} \right)-\Lag \left( x_{*} , \lambda_{k} \right) = o \left( \dfrac{1}{k } \right) \ \mbox{as} \ k \rightarrow +\infty.
\end{align*}
Furthermore, from \eqref{Inq:pre func} and the Cauchy-Schwarz inequality, we have for every $k \geq 1$
\begin{align*}
\left\langle \bz_{k} - \bz_{*} , \bN\bxi_k+\bQ(\bz_{k})\right\rangle 
\geq \ & f\left( x_{k} \right)-f(x_{*}) + h\left( x_{k} \right)-h(x_{*}) + \langle \lambda_{*} , Ax_{k}-b\rangle\\
\geq \ & (f+h) \left( x_{k} \right) - (f+h)(x_{*}) - \left\lVert \lambda_{*}  \right\rVert \left\lVert Ax_{k}-b \right\rVert,
\end{align*}
while on the other hand, since $-A^*\lambda_{*} \in \partial (f+h) \left( x_{*} \right)$ and $Ax_{*} = b$, we deduce
\begin{align*}
(f+h) \left( x_{k} \right) &\geq (f+h) \left( x_{*} \right) - \langle A^*\lambda_{*} , x_{k} - x_{*} \rangle = (f+h)(x_{*}) - \langle \lambda_{*} , Ax_{k} - b \rangle \\
&\geq (f+h)(x_{*}) - \left\lVert \lambda_{*}  \right\rVert \left\lVert Ax_{k}-b \right\rVert.
\end{align*}
Combining these relations, we obtain for every $k \geq 1$
\begin{equation*}
- \left\lVert \lambda_{*} \right\rVert \left\lVert Ax_{k}-b \right\rVert \leq (f+h) \left( x_{k} \right) - (f+h)(x_{*}) \leq \left\lVert \lambda_{*} \right\rVert \left\lVert Ax_{k}-b \right\rVert + \left\langle \bz_{k} - \bz_{*} , \bN\bxi_k+\bQ(\bz_{k})\right\rangle .
\end{equation*}
Multiplying the above relation by $k$, letting $k$ converge to $+ \infty$, and using \eqref{rateAb}, the convergence rate for the objective function value follows.
\end{proof}

\begin{rmk}\label{rem33}
An equivalent formulation for \cref{algo:fpdVu} can be provided by using instead of \eqref{algo:z-y} the iterative scheme
\eqref{algo:z-xi} when solving \eqref{renormedincl}. For initial points $\bz_0, \bz_1 \in \bH$ and $\bxi_1 \in M(\bz_1)$ and step size $\gamma=1$, \eqref{algo:z-xi} reads for every $k \geq 1$
\begin{equation*}\
\left\{
\begin{aligned}\label{a:2222}
& \bz_{k+1}:=J_{\bN^{-1}\bP} \Biggl( \bz_{k} -\bN^{-1}\bQ(\bz_{k}) + \left( 1 - \dfrac{\alpha}{k + \alpha} \right) \left( \bz_{k} - \bz_{k-1} \right) + \frac{2k+\alpha}{2(k+\alpha)} \left(\bxi_k +\bN^{-1}\bQ(\bz_{k-1})\right)\Biggr); \\
& \bxi_{k+1}:=\bz_{k} - \bN^{-1}\bQ(\bz_{k}) + \left( 1 - \dfrac{\alpha}{k + \alpha} \right) ( \bz_{k} - \bz_{k-1} )+ \frac{2k+\alpha}{2(k+\alpha)} \left(\bxi_k  +\bN^{-1}\bQ(\bz_{k-1})\right) -\bz_{k+1},
\end{aligned}\right.
\end{equation*}
which is equivalent to
\begin{equation*}\
\left\{
\begin{aligned}\label{a:equa:zk}
& \bz_{k+1}+\bN^{-1}\bP(\bz_{k+1}) \ni \bz_{k} -\bN^{-1}\bQ(\bz_{k}) + \left( 1 - \dfrac{\alpha}{k + \alpha} \right) \left( \bz_{k} - \bz_{k-1} \right)\\
& \qquad \qquad \qquad \qquad \qquad + \frac{2k+\alpha}{2(k+\alpha)} \left(\bxi_k +\bN^{-1}\bQ(\bz_{k-1})\right);\\
& \bxi_{k+1} =  \bz_{k}  - \bN^{-1}\bQ(\bz_{k}) + \left( 1 - \dfrac{\alpha}{k + \alpha} \right) ( \bz_{k} - \bz_{k-1} )+ \frac{2k+\alpha}{2(k+\alpha)} \left(\bxi_k  +\bN^{-1}\bQ(\bz_{k-1})\right) -\bz_{k+1},\\
\end{aligned}\right.
\end{equation*}
and further to
\begin{equation}\
\left\{
\begin{aligned}\label{a:cf22}
& \bN \bz_{k+1}+\bP(\bz_{k+1})  \ni  \bN\left(\bz_k+\big(1-\frac{\alpha}{k+\alpha}\big)(\bz_k-\bz_{k-1}) +\frac{2k+\alpha}{2(k+\alpha)}\bxi_k \right) +\frac{2k+\alpha}{2(k+\alpha)}\bQ(\bz_{k-1})-\bQ(\bz_k);\\
& \bN\bxi_{k+1} =\bN\left(\bz_k-\bz_{k+1}+\Big(1-\frac{\alpha}{k+\alpha}\Big)(\bz_k-\bz_{k-1}) +\frac{2k+\alpha}{2(k+\alpha)} \bxi_k \right) +\frac{2k+\alpha}{2(k+\alpha)}\bQ(\bz_{k-1})-\bQ(\bz_k).
\end{aligned}\right.
\end{equation}
According to the definitions of $\bN$, $\bP$ and $\bQ$ in \eqref{lef15} and \eqref{def:N}, and using $\bz_k=(x_{k},\lambda_{k})$ and $\bN\bxi_k= (w_k,\zeta_{k})$ in \eqref{cf22},  we obtain an equivalent iterative scheme that reads for every $k \geq 1$
\begin{equation}\
 \left\{
\begin{aligned}\label{a:inequ200}
\frac{1}{\tau}x_{k+1}+\partial f(x_{k+1})
&\ni\frac{1}{\tau} \left( x_{k}+\left(1-\frac{\alpha}{k+\alpha}\right)\left( x_{k}-x_{k-1} \right) \right) - A^*\left(\lambda_{k}+\left(1-\frac{\alpha}{k+\alpha}\right) \left( \lambda_{k}-\lambda_{k-1} \right)\right) \\
& \quad +\frac{2k+\alpha}{2(k+\alpha)} w_{k}  +\frac{2k+\alpha}{2(k+\alpha)} \nabla h\left( x_{k-1} \right) - \nabla h\left( x_{k} \right);\\
\frac{1}{\sigma}\lambda_{k+1}+b-2Ax_{k+1} 
&= - A\left(x_{k}+\left(1-\frac{\alpha}{k+\alpha}\right)\left( x_{k}-x_{k-1} \right)\right) \\
& \quad +\frac{1}{\sigma}\left(\lambda_{k}+\left(1-\frac{\alpha}{k+\alpha}\right)\left( \lambda_{k}-\lambda_{k-1} \right) \right)
+\frac{2k+\alpha}{2(k+\alpha)} \zeta_{k};\\
w_{k+1}			& =\frac{1}{\tau}\left(x_{k}-x_{k+1}+\left(1-\frac{\alpha}{k+\alpha}\right)\left( x_{k}-x_{k-1} \right)\right) \\
& \quad -A^*\left( \lambda_{k}-\lambda_{k+1}+\left(1-\frac{\alpha}{k+\alpha}\right)\left( \lambda_{k}-\lambda_{k-1} \right)\right) \\
& \quad +\frac{2k+\alpha}{2(k+\alpha)}w_k +\frac{2k+\alpha}{2(k+\alpha)}\nabla h\left( x_{k-1} \right)-\nabla h\left( x_{k} \right);\\
\zeta_{k+1} 	& = -A\left(x_{k}-x_{k+1}+\left(1-\frac{\alpha}{k+\alpha}\right)\left( x_{k}-x_{k-1} \right)\right)\\ & \quad +\frac{1}{\sigma}\left(\lambda_{k}-\lambda_{k+1}+\left(1-\frac{\alpha}{k+\alpha}\right)\left( \lambda_{k}-\lambda_{k-1} \right)\right) +\frac{2k+\alpha}{2(k+\alpha)}\zeta_{k}.
\end{aligned}\right.
\end{equation}
According to \eqref{w in sf},
\begin{equation*}
\zeta_{k}= -Ax_{k}+b \quad \forall k \geq 1,
\end{equation*}
which means that we can  formulate \eqref{a:inequ200} by making use of only the first three sequences. Noticing that the first equation in \eqref{a:inequ200} is equivalent to
\begin{align*}
&x_{k+1}=\textnormal{prox}_{\tau f}\Bigg(x_{k}-\tau \left( \nabla h\left( x_{k} \right) + A^* \lambda_{k} \right) +\left(1-\frac{\alpha}{k+\alpha}\right)\left( x_{k}-x_{k-1} \right)-\tau A^*\left(1-\frac{\alpha}{k+\alpha}\right) \left( \lambda_{k}-\lambda_{k-1} \right)\nonumber\\
&\qquad\qquad \qquad\qquad+\frac{2k+\alpha}{2(k+\alpha)}\tau \left( w_{k}+ \nabla h\left( x_{k-1} \right) \right) \Bigg),
\end{align*}
and the second equation in \eqref{a:inequ200}  is equivalent to
\begin{align*}
&\lambda_{k+1}=2\sigma A x_{k+1}-\sigma b-\sigma A\left(x_{k}+\left(1-\frac{\alpha}{k+\alpha}\right)\left( x_{k}-x_{k-1} \right)\right)-\frac{2k+\alpha}{2(k+\alpha)}\sigma \left( Ax_{k} -b \right)\nonumber\\
&\qquad \quad+\lambda_{k}+\left(1-\frac{\alpha}{k+\alpha}\right)\left( \lambda_{k}-\lambda_{k-1} \right),
\end{align*}
we obtain a further alternative formulation for \cref{algo:fpdVu}  that reads for every $k \geq 1$
\begin{equation}\
\left\{
	\begin{aligned}
		x_{k+1}  := & \ \textnormal{prox}_{\tau f}\Bigg(x_{k}-\tau \left( \nabla h\left( x_{k} \right) + A^* \lambda_{k} \right) +\left(1-\frac{\alpha}{k+\alpha}\right)\left( x_{k}-x_{k-1} \right)\nonumber\\
		&\qquad\qquad -\tau A^*\left(1-\frac{\alpha}{k+\alpha}\right) \left( \lambda_{k}-\lambda_{k-1} \right)+\left( 1 - \frac{\alpha}{2(k+\alpha)} \right)\tau \left( w_{k}+ \nabla h\left( x_{k-1} \right) \right) \Bigg);\\
		\lambda_{k+1} := & \ \lambda_{k}+ \sigma \left( Ax_{k+1}-b \right)+\left(1-\frac{\alpha}{k+\alpha}\right)\left( \lambda_{k}-\lambda_{k-1} \right)\\
& + \sigma A\left(x_{k+1}-x_{k} - \left(1-\frac{\alpha}{k+\alpha}\right)\left( x_{k}-x_{k-1} \right)\right)  -\left( 1 - \frac{\alpha}{2(k+\alpha)} \right)\sigma \left( Ax_{k} -b \right);\nonumber\\
		w_{k+1} := & \frac{1}{\tau}\left(x_{k}-x_{k+1}+\left(1-\frac{\alpha}{k+\alpha}\right)\left( x_{k}-x_{k-1} \right)\right) -A^*\left( \lambda_{k}-\lambda_{k+1}+\left(1-\frac{\alpha}{k+\alpha}\right)\left( \lambda_{k}-\lambda_{k-1} \right)\right) \\
&+\left( 1 - \frac{\alpha}{2(k+\alpha)} \right) \left( w_k +\nabla h\left( x_{k-1} \right) \right) -\nabla h\left( x_{k} \right).
	\end{aligned}\right.
\end{equation}
\end{rmk}

\begin{rmk}\label{rem34}
Sabach and Teboulle proposed in \cite{Sabach2022Teboulle} in finite-dimensional spaces a class of \emph{faster Lagrangian-based methods} (FLAG), which, when applied to the convex optimization problem \eqref{convexpro}, for initial points $x_1$ and $\widebar{x}_1$ and constants $\tau,  r > 0$ and $\theta \in(0,1]$, reads
\begin{equation*}\
(\forall k \geq 1) \quad	\left\{
	\begin{aligned}
		& \widebar{x}_{k+1} := \textnormal{prox}_{\tau f}\left(\widebar{x}_{k}-\tau \left( \nabla h(\widebar{x}_k) + A^{*} \lambda_{k} + r A^{*} (A\widebar{x}_{k}-b) \right) - \tau r k A^{*}(Ax_{k}-b) \right); \\
		& \lambda_{k+1} := \lambda_{k} + \theta r \left(A\widebar{x}_{k+1}-b\right);\\
		&x_{k+1} :=\left(1-\frac{1}{k+1}\right)x_{k} +\frac{1}{k+1}\widebar{x}_{k+1}.
	\end{aligned}\right.
\end{equation*}
The authors derived nonergodic/last iterate convergence rates for the function values and the feasibility violation measure, expressed as
\begin{equation*}
	(f+h)\left( x_{k} \right)-(f+h)(x_*) = \mathcal{O} \left( \frac{1}{k} \right)
	\quad \textrm{ and } \quad
	\left\lVert A\left( x_{k} \right)-b \right\rVert=\mathcal{O} \left( \frac{1}{k} \right) \ \mbox{as} \ k \rightarrow +\infty.
\end{equation*}

Recently, He, Huang, and Fang introduced in \cite{He-Huang-Fang} a primal-dual full splitting algorithm in finite-dimensional spaces, formulated in a manner akin to FLAG. When applied to the convex optimization problem \eqref{convexpro}, their algorithm demonstrates nonergodic/last iterate convergence rates of $\mathcal{O} \left( \frac{1}{k} \right)$ for the restricted primal-dual gap, as well as for the primal and dual discrete velocities. Moreover, they establish that every limit point of the primal-dual iterates sequence is a primal-dual optimal solution.

In contrast, our convergence rate results are of a \emph{little-o} type instead of the \emph{Big-O} type. Specifically, this applies to the primal and dual discrete velocities, the primal-dual gap, the feasibility violation measure, and function values. For the latter, we also present a convergence rate from below. Furthermore, we have successfully proven the convergence of the entire primal-dual sequence of iterates to a primal-dual optimal solution. 

Our algorithm demonstrates the most advanced convergence and convergence rate results among primal-dual full splitting algorithms in the context of minimizing a nonsmooth convex function under linear constraints. Moreover, the iterative scheme and its convergence results can be readily extended to minimize separable nonsmooth convex optimization functions with linearly coupled block variables (see \cite{Bot2021Nguyen}).
\end{rmk}

\section{Numerical experiments}\label{sec4}

In this section, we validate the derived convergence rates for our proposed algorithm concerning optimization problems with linear constraints \eqref{convexpro}. We conduct numerical experiments to compare its performance with the FLAG method \cite{Sabach2022Teboulle}. All experiments are executed using Matlab on a standard Lenovo laptop equipped with an Intel(R) Core(TM) i5-8265U CPU and 8GB RAM.

\begin{figure}[ht!]
	\centering
	\begin{subfigure}[b]{\textwidth}
		\centering
		\includegraphics[width=0.50\linewidth]{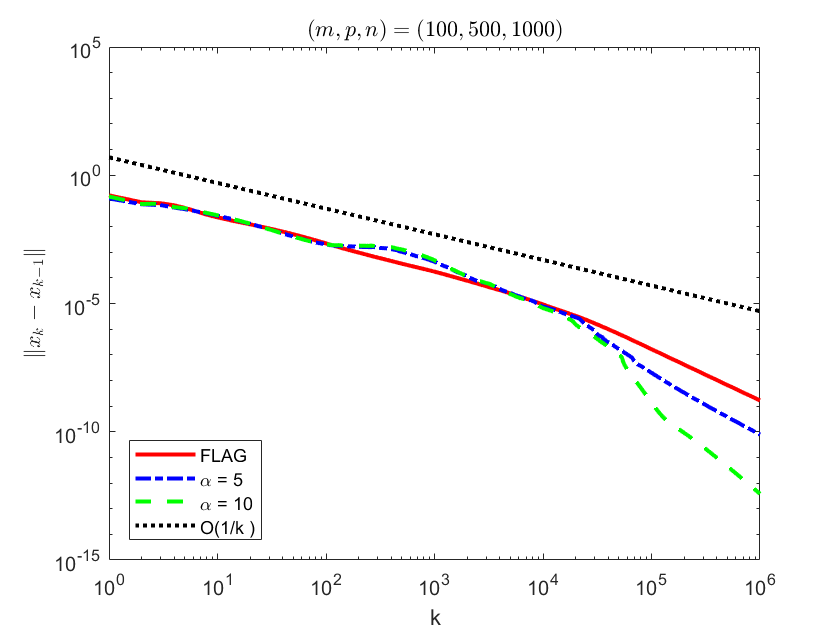}%
		\hfill
		\includegraphics[width=0.50\linewidth]{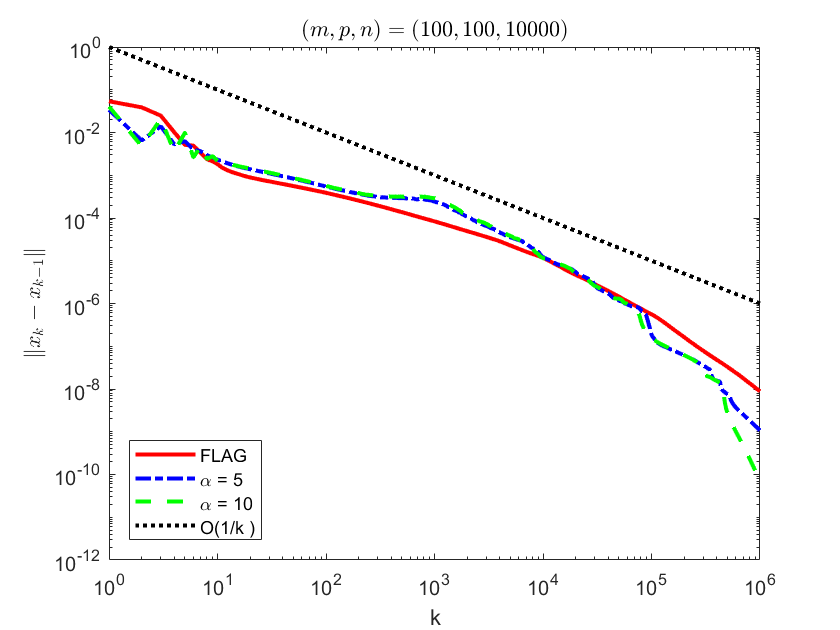}
		\caption{The discrete velocities of the primal sequence}
	\end{subfigure}
	\vskip\baselineskip
	\begin{subfigure}[b]{\textwidth}
		\centering
		\includegraphics[width=0.50\linewidth]{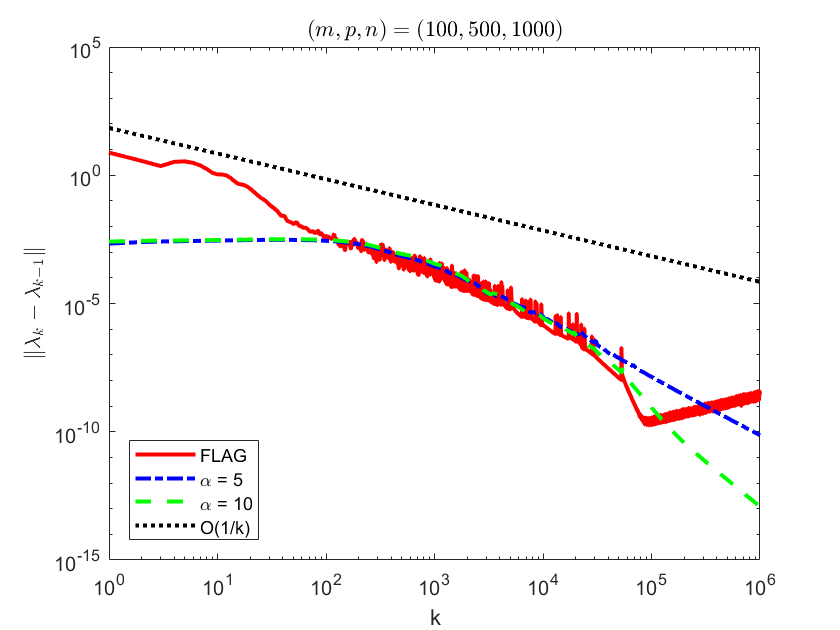}%
		\hfill
		\includegraphics[width=0.50\linewidth]{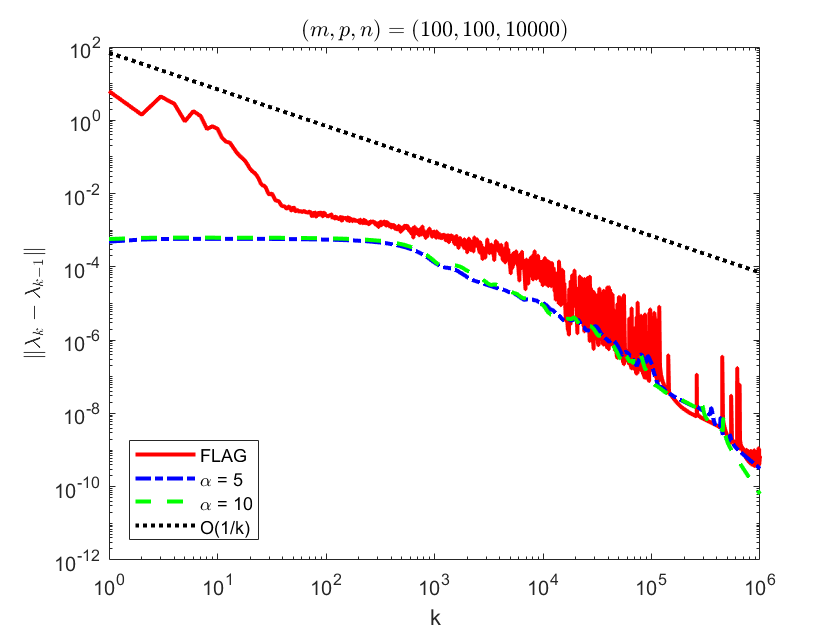}
		\caption{The discrete velocities of the dual sequence}
	\end{subfigure}
	\caption{Comparing the convergence behaviour of the discrete velocities across various dimension settings}\label{figure41}
\end{figure}

\begin{table}[h!]\normalsize
	\centering	
	\begin{tabular}{|c|l|l|l|}
		\hline
		\diagbox{Quantity}{Measure} & \multicolumn{1}{c|}{\textrm{FLAG}} & \multicolumn{1}{c|}{$\alpha = 5$} 	& \multicolumn{1}{c|}{$\alpha = 10$} \\
		\hline
		$\left\lVert x_{k}-x_{k-1} \right\rVert$ 				&	 $1.5124\times 10^{-9}$ 	& $1.2723\times 10^{-10}$ 	& $3.6571\times 10^{-13}$\\
		\hline
		$\left\lVert \lambda_{k}-\lambda_{k-1} \right\rVert$ 	&	 $2.4809\times 10^{-9}$ 	& $ 6.2010\times 10^{-11}$ 	& $1.7187\times 10^{-13}$\\
		\hline	
		$\left\lVert Ax_{k}-b \right\rVert$ 					&	 $3.6996\times 10^{-7}$ 	& $3.4758\times 10^{-4}$ 	& $3.3808\times 10^{-7}$\\
		\hline
		$\left( f+h \right) \left( x_{k} \right)$ 				&	 $25.6340$ 					& $25.6330$ 				& $25.6332$\\		
		\hline
	\end{tabular}
	\caption{Comparing the average convergence behavior of the discrete velocities, feasibility violation measure, and function values for $\left( m, p, n \right) = \left( 100, 500, 1000 \right)$.}\label{inn222}
\end{table}

Particulaly, we considered the following convex optimization problem
\begin{align*}
\min_{x\in \mathbb{R}^{n}} \,\,& \left\lVert x \right\rVert _{1} +\frac{1}{2} \left\lVert Bx-c \right\rVert ^{2} \nonumber\\
\textnormal{s.t.}\,\, &Ax=b,
\end{align*}
where $A\in \mathbb{R}^{m\times n}$, $B\in \mathbb{R}^{p\times n}$, $b\in \mathbb{R}^{m}$ and $c\in \mathbb{R}^{p}$. We carried out the numerical experiments for different choices of the triple of dimensions $(m,p,n)$ for randomly generated matrices and vectors.  For each dimensional setting,  we conducted the experiments $10$ times with $10^6$ iterations per run.  This was achieved by generating
 $A\in \mathbb{R}^{m\times n}$, $B\in \mathbb{R}^{p\times n}$, $b\in \mathbb{R}^{m}$ and $c\in \mathbb{R}^{p}$, all following a standard normal distribution.  In the Tables \ref{inn222} and \ref{inn2222} and the Figures  \ref{figure41} and  \ref{figure42} we showcase the average convergence behavior of the discrete velocities, feasibility violation measure, and function values over $10$ runs for $(m,p,n) = (100,500,1000)$ and $(m,p,n)=(100,100,10000)$, respectively.  These two instances serve as typical examples for the convergence behavior of the two considered numerical methods.  

When solving the optimization problems with FLAG and \cref{algo:fpdVu}, we selected optimal algorithm parameters based on theoretical considerations.  According to \cite{Sabach2022Teboulle}, we chose
\begin{equation*}
	r = 1 \quad \textrm{ and } \quad \tau = \dfrac{\beta}{\beta \left\lVert A \right\rVert ^{2} + 1} .
\end{equation*}
Specifically for \cref{algo:fpdVu}, we explored different choices for $\alpha$ and chose accordingly
\begin{equation*}
	\tau = \sigma := \dfrac{0.99 \beta}{\beta \left\lVert A \right\rVert + \left( 1 - \frac{0.99 \cdot 3 \left( \alpha - 2 \right)}{8 \left( \alpha - 1 \right)} \right)} < \dfrac{\beta}{\beta \left\lVert A \right\rVert + \left( 1 - \frac{3 \left( \alpha - 2 \right)}{8 \left( \alpha - 1 \right)} \right)} = \dfrac{1}{\left\lVert A \right\rVert + \frac{5 \alpha - 2}{8 \left( \alpha - 1 \right) \beta}} .
\end{equation*}
\color{black}

FLAG demonstrates superior convergence concerning the feasibility violation measure, while \cref{algo:fpdVu} outperforms FLAG in terms of the convergence of function values and discrete velocities.

\begin{figure}[ht!]
	\centering
	\begin{subfigure}[b]{\textwidth}
		\centering
		\includegraphics[width=0.50\linewidth]{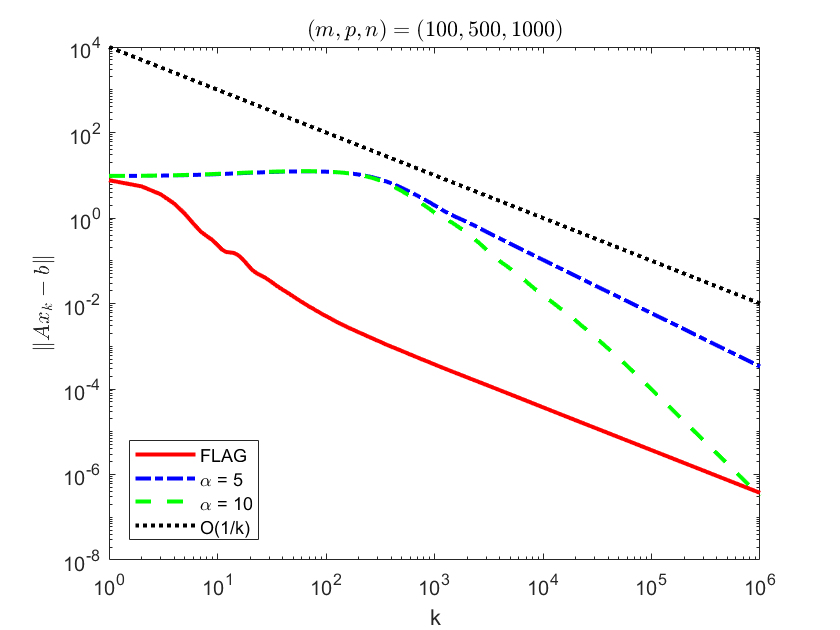}%
		\hfill
		\includegraphics[width=0.50\linewidth]{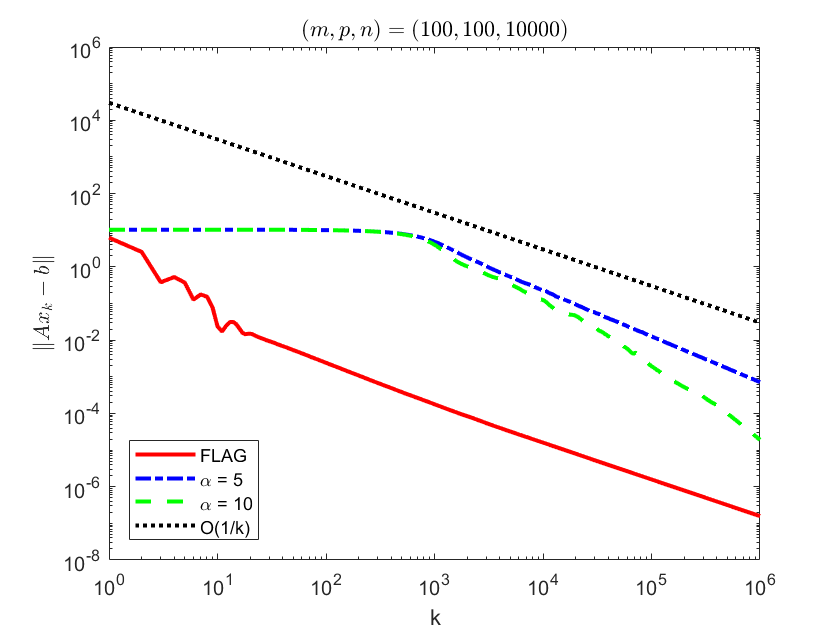}
		\caption{The feasibility violation measures}
	\end{subfigure}
	\vskip\baselineskip
	\begin{subfigure}[b]{\textwidth}
		\centering
		\includegraphics[width=0.50\linewidth]{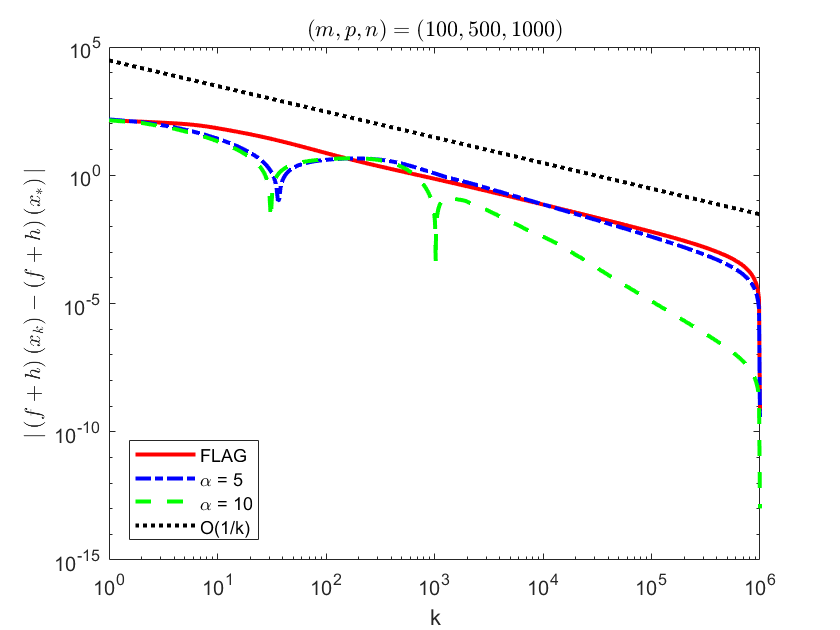}%
		\hfill
		\includegraphics[width=0.50\linewidth]{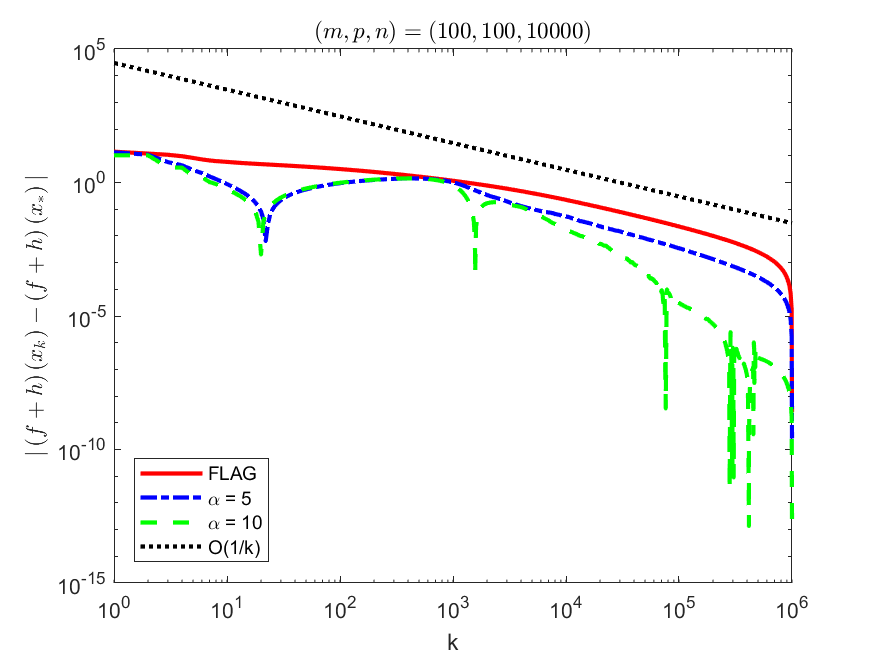}
		\caption{The function values}
	\end{subfigure}
	\caption{Comparing the convergence behavior of the feasibility violation measure and function values across various dimension settings}\label{figure42}
\end{figure}


\begin{table}[h!]\normalsize
	\centering	
	\begin{tabular}{|c|l|l|l|}
		\hline
		\diagbox{Quantity}{Measure} & \multicolumn{1}{c|}{\textrm{FLAG}} & \multicolumn{1}{c|}{$\alpha = 5$} 	& \multicolumn{1}{c|}{$\alpha = 10$} \\
		\hline
		$\left\lVert x_{k}-x_{k-1} \right\rVert$ 				&	 $1.2429\times 10^{-9}$ 	& $3.0553\times 10^{-10}$ 	& $1.5930\times 10^{-13}$\\
		\hline
		$\left\lVert \lambda_{k}-\lambda_{k-1} \right\rVert$ 	&	 $5.6672\times 10^{-10}$ 	& $1.6168\times 10^{-9}$ 	& $4.1257\times 10^{-10}$\\
		\hline	
		$\left\lVert Ax_{k}-b \right\rVert$ 					&	 $1.5354\times 10^{-7}$ 	& $7.0237\times 10^{-4}$ 	& $2.4590\times 10^{-5}$\\
		\hline
		$\left( f+h \right) \left( x_{k} \right)$ 				&	 $5.2996$ 					& $5.2967$ 				& $5.2969$\\		
		\hline
	\end{tabular}
	\caption{Comparing the average convergence behavior of the discrete velocities, feasibility violation measure, and function values for $\left( m, p, n \right) = \left( 100, 100, 10000 \right)$.}\label{inn2222}
\end{table}


\pagebreak
\appendix
\section{Appendix}

Within the appendix, we have assembled supplementary results and furnished proofs for the two technical lemmas essential in analyzing the convergence of the proposed forward-backward algorithm.

\subsection{Preliminaries}\label{subseca1}

\begin{lem}\label{lem111}
Let $a,b,c\in \mathbb{R}$ be such that $a\neq 0$ and $b^2-ac\leq 0$. The following statements are true:
\begin{enumerate}
		\item
		if $a>0$, then
\begin{equation}
a\|x\|^2+2b\langle x,y\rangle+c\|y\|^2\geq 0\qquad \forall x,y\in \mathcal{H};\nonumber
\end{equation}
		
		\item
		if $a<0$, then
\begin{equation}
a\|x\|^2+2b\langle x,y\rangle+c\|y\|^2\geq 0\qquad \forall x,y\in \mathcal{H}.\nonumber
\end{equation}
	\end{enumerate}
\end{lem}

The following result follows from \cite[Lemma 5.31]{bauschkebook2011}.

\begin{lem}\label{ddec}
Let $(a_k)_{k\geq 1}$, $(b_k)_{k\geq 1}$ and $(d_k)_{k\geq 1}$ be sequences of real numbers. If $(a_k)_{k\geq 1}$ is bounded from below, $(b_k)_{k\geq 1}$ and $(d_k)_{k\geq 1}$ are nonnegative sequences satisfying $\sum_{k\geq 1}d_k<+\infty$ and
\begin{equation}\label{inequ:13}
a_{k+1}\leq (1+d_k)a_k -b_k \quad \forall k\geq 1,
\end{equation}
then the following statements are true:
\begin{enumerate}
		\item
		the sequence $(b_k)_{k\geq 1}$ is summable, i.e., $\sum_{k\geq 1}b_k<+\infty$;
		
		\item
		the sequence $(a_k)_{k\geq 1}$ is convergent.
	\end{enumerate}

\end{lem}

The following result is stated and proven in \cite[Lemma A.5]{Bot2023Robert_Nguyen}.

\begin{lem}
	\label{lem:lim-u-k}
	Let $a \geq 1$ and $\left( q_{k} \right)_{k \geq 0}$ be a bounded sequence in $\sH$ such that
	\begin{equation*}
		\lim\limits_{k \to + \infty} \left( q_{k+1} + \dfrac{k}{a} \left( q_{k+1} - q_{k} \right) \right) = l \in \sH .
	\end{equation*}
	Then it holds $\lim\limits_{k \to + \infty} q_{k} = l$.
\end{lem}

For the convergence proof of the iterates we use the discrete version of Opial’s Lemma.

\begin{lem}
	\label{lem:Opial:dis}
	Let $\mathcal{S}$ be a nonempty subset of $\sH$ and $\left( z_{k} \right) _{k \geq 1}$ be a sequence in $\sH$.
	Assume that
	\begin{enumerate}
		\item
		\label{lem:Opial:dis:i}
		for every $z_{*} \in \mathcal{S}$, $\lim\limits_{k \to + \infty} \left\lVert z_{k} - z_{*} \right\rVert$ exists;
		
		\item
		\label{lem:Opial:dis:ii}
		every weak sequential cluster point of the sequence $\left(z_{k} \right) _{k \geq 1}$ as $k \to + \infty$ belongs to $\mathcal{S}$.
	\end{enumerate}
	Then $\left(z_{k} \right)_{k \geq 1}$ converges weakly to a point in $\mathcal{S}$ as $k \to + \infty$.
\end{lem}

The following lemma illustrates the relationship between the two convergence measures utilized in this paper (see [Fact 1, \cite{Yang2022Zheng}]).

\begin{lem}\label{bb1}
For every $z\in \sH$ and every $\gamma >0$, it holds
\begin{equation*}
 r_{\mathrm{fix}, \gamma}(z) = \|z-J_{\gamma M}\big(z-\gamma C(z)\big)\| \leq \gamma \inf_{\xi\in M(z)}\|\xi+C(z)\| = \gamma \dist(0, M(z) + C(z)) = \gamma r_{\mathrm{tan}}(z). 
\end{equation*}
\end{lem}
\begin{proof}
Let  $z\in \sH$. For arbitrary $\xi\in M(z)$, it follows from the nonexpansivity of $J_{\gamma M}$ that
\begin{align*}
	r_{\mathrm{fix}}(z)&=\|z-J_{\gamma M}\big(z-\gamma C(z)\big)\| =\|J_{\gamma M}(z+ \gamma \xi)-J_{\gamma M}\big(z-\gamma C(z)\big)\| \leq \gamma \|\xi+C(z)\|.
\end{align*}
This proves the statement.
\end{proof}

\subsection{Proofs of the technical lemmas}\label{subseca2}

In this subsection we provide the proofs of the Lemmas \ref{lemma8}, \ref{lemma18}, and \ref{lemm2:3}. By invoking the second update block in Algorithm \ref{algo:ffb}, we have that for every $k \geq 1$
\begin{align}
 & \ 2 \left( k + \alpha \right) \left( z_{k+1} - z_{k} \right) \nonumber \\
= &  \ 2k \left( z_{k} - z_{k-1} \right) -  (2 k+\alpha) \gamma  \Big( \xi_{k+1}+C \left(z_k\right) - \left( \xi_{k}+C \left( z_{k-1} \right) \right)\Big) - \gamma \alpha  \left(\xi_{k+1}+C \left(z_k\right)\right), \label{dis:d-u}
\end{align}

\begin{proofff}
Let $0 \leq \eta \leq \alpha-1$.  For brevity  we denote for every $k \geq 0$
\begin{equation}
\label{defi:u-k-lambda}
u_{\eta,k+1} := 2 \eta \left( z_{k+1} - z_{*} \right) + 2 \left( k + 1 \right) \left( z_{k+1} - z_{k} \right) +  \frac{5\alpha-2}{4(\alpha-1)}\gamma \left( k + 1 \right) (\xi_{k+1}+C \left( z_{k} \right)) .
\end{equation}
This means that  for every $k \geq 1$ it holds
\begin{equation*}
u_{\eta,k} = 2 \eta \left( z_{k} - z_{*} \right) + 2k \left( z_{k} - z_{k-1} \right) + \frac{5\alpha-2}{4(\alpha-1)}\gamma k  (\xi_{k}+C \left( z_{k-1} \right)).
\end{equation*}
Therefore, taking the difference and using \eqref{dis:d-u}, we deduce that for every $k \geq 1$
\begin{align}
u_{\eta,k+1} - u_{\eta,k}
= & \ 2 \left( \eta + 1 - \alpha \right) \left( z_{k+1} - z_{k} \right) + 2 \left( k + \alpha \right) \left( z_{k+1} - z_{k} \right) - 2k \left( z_{k} - z_{k-1} \right) \nonumber \\
& + \frac{5\alpha-2}{4(\alpha-1)}\gamma \left( k + 1 \right) (\xi_{k+1}+C \left( z_{k} \right)) - \frac{5\alpha-2}{4(\alpha-1)}\gamma k  (\xi_{k}+C \left( z_{k-1} \right)) \nonumber \\
= & \ 2 \left( \eta + 1 - \alpha \right) \left( z_{k+1} - z_{k} \right) - \frac{(4\alpha -1)(\alpha-2)}{4(\alpha -1)} \gamma \left(\xi_{k+1}+C \left( z_{k} \right)\right)\nonumber\\
&- \left(\frac{3(\alpha-2)}{4(\alpha-1)}k+\alpha\right)\gamma \Big(\xi_{k+1}+C \left( z_{k} \right)-(\xi_{k}+C \left( z_{k-1} \right)) \Big). \label{defi:u-k-lambda:dif}
\end{align}
Let $k \geq 1$. In the following we want to use the identity
\begin{equation}
\label{dec:dif:u-lambda:pre}
\dfrac{1}{2} \left( \left\lVert u_{\eta,k+1} \right\rVert ^{2} - \left\lVert u_{\eta,k} \right\rVert ^{2} \right) = \left\langle u_{\eta,k+1} , u_{\eta,k+1} - u_{\eta,k} \right\rangle - \dfrac{1}{2} \left\lVert u_{\eta,k+1} - u_{\eta,k} \right\rVert ^{2} .
\end{equation}
Using \eqref{defi:u-k-lambda} and \eqref{defi:u-k-lambda:dif}, we obtain that
\begin{align}
& \left\langle u_{\eta,k+1} , u_{\eta,k+1} - u_{\eta,k} \right\rangle \nonumber \\
= & \ 4 \eta \left( \eta + 1 - \alpha \right) \left\langle z_{k+1} - z_{*} , z_{k+1} - z_{k} \right\rangle
-\frac{(4\alpha-1)(\alpha-2)}{2(\alpha-1)} \eta \gamma \left\langle z_{k+1} - z_{*} , \xi_{k+1}+C \left( z_{k} \right) \right\rangle \nonumber \\
& - 2\left(\frac{3(\alpha -2)}{4(\alpha -1)}k+\alpha\right)\eta \gamma  \left\langle z_{k+1} - z_{*} , \xi_{k+1}+C \left( z_{k} \right) - (\xi_{k}+C \left( z_{k-1} \right)) \right\rangle
\nonumber \\
& + \left(\frac{5\alpha-2}{2(\alpha-1)}(\eta+1-\alpha)-\frac{(4\alpha-1)(\alpha-2)}{2(\alpha-1)}\right) \gamma \left( k + 1 \right)  \left\langle z_{k+1} - z_{k} , \xi_{k+1}+C \left( z_{k} \right) \right\rangle \nonumber \\
&  - 2\left(\frac{3(\alpha -2)}{4(\alpha -1)}k+\alpha\right) \gamma \left( k + 1 \right) \left\langle z_{k+1} - z_{k} , \xi_{k+1}+C \left( z_{k} \right) - (\xi_{k}+C \left( z_{k-1} \right)) \right\rangle\nonumber\\
&  -  \frac{(4\alpha -1)(\alpha-2)(5\alpha-2)}{16(\alpha-1)^2}\gamma^{2} \left( k + 1 \right)  \left\lVert \xi_{k+1}+C \left( z_{k} \right) \right\rVert ^{2}+ 4 \left( \eta + 1 - \alpha \right) \left( k + 1 \right) \left\lVert z_{k+1} - z_{k} \right\rVert ^{2} \nonumber \\
&  - \frac{5\alpha-2}{4(\alpha-1)}\left(\frac{3(\alpha-2)}{4(\alpha-1)}k+\alpha\right) \gamma^{2} \left( k + 1 \right)  \left\langle \xi_{k+1}+C \left( z_{k} \right) , \xi_{k+1}+C \left( z_{k} \right)- (\xi_{k}+C \left( z_{k-1} \right))\right\rangle , \label{dec:dif:u-lambda:inn}
\end{align}
and
\begin{align}
&- \dfrac{1}{2} \left\lVert u_{\eta,k+1} - u_{\eta,k} \right\rVert ^{2}\nonumber\\
= &-2\left( \eta + 1 - \alpha \right)^2\|z_{k+1}-z_{k}\|^2- \frac{(4\alpha-1)^2(\alpha-2)^2}{32(\alpha-1)^2} \gamma^{2} \left\lVert \xi_{k+1}+C \left( z_{k} \right) \right\rVert ^{2} \nonumber \\
& \quad - \frac{1}{2}\left(\frac{3(\alpha-2)}{4(\alpha-1)}k+\alpha\right)^2 \gamma^{2}  \left\lVert \xi_{k+1}+C \left( z_{k} \right) - (\xi_{k}+C \left( z_{k-1} \right)) \right\rVert ^{2} \nonumber \\
& \quad +\frac{(4\alpha-1)(\alpha-2)}{2(\alpha-1)}\left( \eta + 1 - \alpha \right)  \gamma  \left\langle z_{k+1} - z_{k} , \xi_{k+1}+C \left( z_{k} \right) \right\rangle \nonumber \\
& \quad	+2\left( \eta + 1 - \alpha \right) \left(\frac{3(\alpha-2)}{4(\alpha-1)} k+\alpha\right) \gamma \left\langle z_{k+1} - z_{k} , \xi_{k+1}+C \left( z_{k} \right) - (\xi_{k}+C \left( z_{k-1} \right)) \right\rangle \nonumber \\
& \quad -\frac{(4\alpha-1)(\alpha-2)}{4(\alpha-1)}\left(\frac{3(\alpha-2)}{4(\alpha-1)}k+\alpha\right) \gamma^{2}    \left\langle \xi_{k+1}+C \left( z_{k} \right) , \xi_{k+1}+C \left( z_{k} \right)- (\xi_{k}+C \left( z_{k-1} \right)) \right\rangle . \label{dec:dif:u-lambda:norm}
\end{align}

By plugging \eqref{dec:dif:u-lambda:inn} and \eqref{dec:dif:u-lambda:norm} into \eqref{dec:dif:u-lambda:pre}, we get 
\begin{align}
& \dfrac{1}{2} \left( \left\lVert u_{\eta,k+1} \right\rVert ^{2} - \left\lVert u_{\eta,k} \right\rVert ^{2} \right) \nonumber \\
= & \ 4 \eta \left( \eta + 1 - \alpha \right) \left\langle z_{k+1} - z_{*} , z_{k+1} - z_{k} \right\rangle
- \frac{(4\alpha-1)(\alpha-2)}{2(\alpha-1)} \eta \gamma \left\langle z_{k+1} - z_{*} , \xi_{k+1}+C \left( z_{k} \right) \right\rangle \nonumber \\
& -2 \left(\frac{3(\alpha-2)}{4(\alpha-1)}k+\alpha\right) \eta \gamma \left\langle z_{k+1} - z_{*} , \xi_{k+1}+C \left( z_{k} \right) - (\xi_{k}+C \left( z_{k-1} \right)) \right\rangle\nonumber\\
&+ 2 \left( \eta + 1 - \alpha \right) \left( 2k + \alpha + 1 - \eta \right) \left\lVert z_{k+1} - z_{k} \right\rVert ^{2} \nonumber \\
& + 2  \left(\frac{5\alpha-2}{4(\alpha-1)} \left( \eta + 1 - \alpha \right) -\frac{(4\alpha-1)(\alpha-2)}{4(\alpha-1)}\right)k\gamma 
\left\langle z_{k+1} - z_{k} , \xi_{k+1}+C \left( z_{k} \right) \right\rangle \nonumber \\
&-2 \left( \frac{(4\alpha-1)(\alpha-2)}{4(\alpha-1)}
-\alpha(\eta+1-\alpha)\right)
\gamma \left\langle z_{k+1} - z_{k} , \xi_{k+1}+C \left( z_{k} \right) \right\rangle \nonumber \\
& - 2 \left(\frac{3(\alpha-2)}{4(\alpha-1)}k+\alpha\right) \left( k + \alpha - \eta \right)\gamma \left\langle z_{k+1} - z_{k} , \xi_{k+1}+C \left( z_{k} \right) - (\xi_{k}+C \left( z_{k-1} \right)) \right\rangle\nonumber\\
& -\frac{(4\alpha-1)(\alpha-2)}{8(\alpha-1)}\left(\frac{5\alpha-2}{2(\alpha-1)}k + \frac{4\alpha^2+\alpha-2}{4(\alpha-1)} \right) \gamma^{2} \left\lVert \xi_{k+1}+C \left( z_{k} \right) \right\rVert ^{2} \nonumber \\
& - \left(\frac{3(\alpha-2)}{4(\alpha-1)}k+\alpha\right) \left(\frac{5\alpha-2}{4(\alpha-1)} k + \alpha \right)  \gamma^{2}  \left\langle \xi_{k+1}+C \left( z_{k} \right) , \xi_{k+1}+C \left( z_{k} \right) - (\xi_{k}+C \left( z_{k-1} \right)) \right\rangle\nonumber\\
&- \frac{1}{2}\left(\frac{3(\alpha-2)}{4(\alpha-1)}k+\alpha\right)^2 \gamma^{2} \left\lVert \xi_{k+1}+C \left( z_{k} \right) - (\xi_{k}+C \left( z_{k-1} \right)) \right\rVert ^{2} . \label{dec:dif:u-lambda}
\end{align}
Above we obtained an estimate for the difference of the first terms in the definition of $\E_{\eta,k+1}$ and $\E_{\eta,k}$, respectively. Next we evaluate the differences of the remaining terms in the two discrete energy functions, respectively.  First, we observe that 
\begin{align}
& \ 2 \eta \left( \alpha - 1 - \eta \right) \left( \left\lVert z_{k+1} - z_{*} \right\rVert ^{2} - \left\lVert z_{k} - z_{*} \right\rVert ^{2} \right) \nonumber \\
= & \  2 \eta \left( \alpha - 1 - \eta \right) \left( 2 \left\langle z_{k+1} - z_{*} , z_{k+1} - z_{k} \right\rangle - \left\lVert z_{k+1} - z_{k} \right\rVert ^{2} \right) . \label{dec:dif:norm}
\end{align}
Further, we have
\begin{align}
& \ 2 \eta \gamma \left(\frac{3(\alpha-2)}{4(\alpha-1)}\left( k + 1 \right)+\alpha\right)   \left\langle z_{k+1} - z_{*} , \xi_{k+1}+C \left( z_{k} \right) \right\rangle\nonumber\\
&- 2 \eta \gamma \left(\frac{3(\alpha-2)}{4(\alpha-1)}k+\alpha\right)  \left\langle z_{k} - z_{*} , \xi_{k}+C \left( z_{k-1} \right) \right\rangle \nonumber \\
=  & \ \frac{3(\alpha-2)}{2(\alpha-1)} \eta \gamma \left\langle z_{k+1} - z_{*} , \xi_{k+1}+C \left( z_{k} \right) \right\rangle + 2 \eta \gamma \left(\frac{3(\alpha-2)}{4(\alpha-1)}k+\alpha\right) \big( \left\langle z_{k+1} - z_{*} , \xi_{k+1}+C \left( z_{k} \right) \right\rangle\nonumber\\
& - \left\langle z_{k} - z_{*} , \xi_{k}+C \left( z_{k-1} \right)\right\rangle \big) \nonumber \\
= & \ \frac{3(\alpha-2)}{2(\alpha-1)} \eta \gamma \left\langle z_{k+1} - z_{*} , \xi_{k+1}+C \left( z_{k} \right) \right\rangle+2 \eta \gamma \left(\frac{3(\alpha-2)}{4(\alpha-1)}k+\alpha\right)  \left\langle z_{k+1} - z_{k} ,  \xi_{k+1}+C \left( z_{k} \right) \right\rangle \nonumber \\
&  + 2 \eta \gamma \left(\frac{3(\alpha-2)}{4(\alpha-1)}k+\alpha\right)  \left\langle z_{k+1} - z_{*} ,  \xi_{k+1}+C \left( z_{k} \right) - (\xi_{k}+C \left( z_{k-1} \right)) \right\rangle\nonumber\\
&- 2 \eta \gamma \left(\frac{3(\alpha-2)}{4(\alpha-1)}k+\alpha\right)   \left\langle z_{k+1} - z_{k} ,  \xi_{k+1}+C \left( z_{k} \right) - (\xi_{k}+C \left( z_{k-1} \right)) \right\rangle, \label{dec:dif:vi}
\end{align}
and
\begin{align}
& \ \dfrac{1}{2} \gamma^{2}\left(\frac{3(\alpha-2)}{4(\alpha-1)}(k+1)+\alpha\right)  \left( \frac{5\alpha-2}{4(\alpha-1)}(k+1) + \alpha \right) \left\lVert \xi_{k+1}+C \left( z_{k} \right) \right\rVert ^{2} \nonumber\\
&- \dfrac{1}{2} \gamma^{2}\left(\frac{3(\alpha-2)}{4(\alpha-1)}k+\alpha\right) \left( \frac{5\alpha-2}{4(\alpha-1)}k + \alpha \right)   \left\lVert \xi_{k}+C \left( z_{k-1} \right) \right\rVert ^{2} \nonumber \\
= & \  \dfrac{1}{2} \gamma^{2} \left( \frac{3(5\alpha-2)(\alpha-2)}{8(\alpha-1)^2}k + \frac{(5\alpha-2)\alpha}{4(\alpha-1)}+\frac{3(\alpha-2)(4\alpha^2+\alpha-2)}{16(\alpha-1)^2} \right) \left\lVert\xi_{k+1}+C \left( z_{k} \right) \right\rVert ^{2}\nonumber\\
&-\dfrac{1}{2} \gamma^{2} \left(\frac{3(\alpha-2)}{4(\alpha-1)}k+\alpha\right)\left( \frac{5\alpha-2}{4(\alpha-1)}k + \alpha \right) \left\lVert \xi_{k+1}+C \left( z_{k} \right)-(\xi_{k}+C \left( z_{k-1} \right)) \right\rVert ^{2} \nonumber \\
& + \gamma^{2}\left(\frac{3(\alpha-2)}{4(\alpha-1)}k+\alpha\right) \left(\frac{5\alpha-2}{4(\alpha-1)}k + \alpha \right) \left\langle \xi_{k+1}+C \left( z_{k} \right),\xi_{k+1}+C \left( z_{k} \right)-(\xi_{k}+C \left( z_{k-1} \right))\right\rangle. \label{dec:dif:eq}
\end{align}
Summing the indentities in  \eqref{dec:dif:u-lambda}, \eqref{dec:dif:norm}, \eqref{dec:dif:vi} and \eqref{dec:dif:eq} and taking into account that $\eta+1-\alpha \leq 0$, we obtain the statement in \cref{lemma8}.
\end{proofff}

\begin{prooffff}
Let $k \geq 1$ be fixed.

(i) By \cref{lemma8}, we have that
\begin{align}
\E_{\eta,k+1} - \E_{\eta,k}
\leq \ & 2 \eta \gamma(2-\alpha)\left\langle z_{k+1} - z_{*} , \xi_{k+1}+C \left( z_{k} \right) \right\rangle+ 2 \gamma \left(\nu_{0}k +\nu_{1}\right) \left\langle z_{k+1} - z_{k} , \xi_{k+1}+C \left( z_{k} \right) \right\rangle \nonumber \\
& + 2\nu_{2} k \left\lVert z_{k+1} - z_{k} \right\rVert ^{2} +\frac{1}{2} \gamma^{2} (\nu_{3}k +\nu_{4}) \left\lVert \xi_{k+1}+C \left( z_{k} \right) \right\rVert ^{2} \nonumber \\
& - 2 \gamma\left(\frac{3(\alpha-2)}{4(\alpha-1)}k+\alpha\right) \left( k + \alpha \right) \left\langle z_{k+1} - z_{k} ,\xi_{k+1}+C \left( z_{k} \right) - (\xi_{k}+C \left( z_{k-1} \right))\right\rangle\nonumber\\
&- \gamma^{2}\left(\frac{3(\alpha-2)}{4(\alpha-1)}k+\alpha\right)(k+\alpha)  \left\lVert \xi_{k+1}+C \left( z_{k} \right) - (\xi_{k}+C \left( z_{k-1} \right)) \right\rVert ^{2}.
\label{dec:inquen}
\end{align}
Using the Cauchy-Schwarz inequality, we obtain
\begin{align}
&-2 \eta \gamma(\alpha-2)\left\langle z_{k+1} - z_{*} , \xi_{k+1}+C \left( z_{k} \right) \right\rangle\nonumber\\
=\,\,& -2 \eta \gamma(\alpha-2)\left\langle z_{k+1} - z_{*} , \xi_{k+1}+C \left( z_{k+1} \right) \right\rangle+2 \eta \gamma(\alpha-2)\left\langle z_{k+1} - z_{*} , C \left( z_{k+1} \right)-C \left( z_{k} \right) \right\rangle\nonumber\\
\leq\,\,& -2 \eta \gamma(\alpha-2)\left\langle z_{k+1} - z_{*} , \xi_{k+1}+C \left( z_{k+1} \right) \right\rangle +\frac{4 \eta^{2} (\alpha -2) (\alpha-1)}{3(k+1)\sqrt{k+1}} \left\lVert z_{k+1}-z_{*} \right\rVert^2\nonumber\\
&\qquad+\frac{3(\alpha -2)}{4(\alpha-1)}\gamma^{2}(k+1)\sqrt{k+1}\|C \left( z_{k+1} \right)-C \left( z_{k} \right)\|^2.
\label{decwb:inq}
\end{align}

Using the monotonicity of $M$ in combination with \eqref{algo:inc-M}, the coercivity of $C$, and \eqref{dis:d-u}, it yields
\begin{align}
&\quad - 2 \gamma\left(\frac{3(\alpha-2)}{4(\alpha-1)}k+\alpha\right) \left( k + \alpha \right) \left\langle z_{k+1} - z_{k} ,\xi_{k+1}+C \left( z_{k} \right) - (\xi_{k}+C \left( z_{k-1} \right)) \right\rangle \nonumber\\
&\,\,=- 2 \gamma\left(\frac{3(\alpha-2)}{4(\alpha-1)}k+\alpha\right) \left( k + \alpha \right)\Big( \left\langle z_{k+1} - z_{k} , \xi_{k+1} - \xi_{k}\right\rangle+ \left\langle z_{k+1} - z_{k} , C \left( z_{k} \right) - C \left( z_{k-1} \right)\right\rangle\Big)\nonumber\\
&\,\,\leq- 2 \gamma\left(\frac{3(\alpha-2)}{4(\alpha-1)}k+\alpha\right)\left( k + \alpha \right) \left\langle z_{k+1} - z_{k} , C \left( z_{k} \right) - C \left( z_{k-1} \right)\right\rangle\nonumber\\
&\,\,=- 2 \gamma\left(\frac{3(\alpha-2)}{4(\alpha-1)}k+\alpha\right)k \left\langle z_{k} - z_{k-1} , C \left( z_{k} \right) - C \left( z_{k-1} \right)\right\rangle\nonumber\\
&\qquad+ \alpha \gamma^{2}\left(\frac{3(\alpha-2)}{4(\alpha-1)}k+\alpha\right)  \left\langle \xi_{k+1}+C \left( z_{k} \right) , C \left( z_{k} \right) - C \left( z_{k-1} \right)\right\rangle\nonumber\\
&\qquad + \gamma^{2}\left(\frac{3(\alpha-2)}{4(\alpha-1)}k+\alpha\right)(2k+\alpha)\left\langle \xi_{k+1}+C \left( z_{k} \right) - (\xi_{k}+C \left( z_{k-1} \right)), C \left( z_{k} \right) - C \left( z_{k-1} \right)\right\rangle\nonumber\\
&\,\,\leq \Bigg(2 \gamma\left(\frac{3(\alpha-2)}{4(\alpha-1)}(k+1)+\alpha\right)(k+1)  \left\langle z_{k+1} - z_{k} , C \left( z_{k+1} \right) - C \left( z_{k} \right)\right\rangle\nonumber\\
&\qquad-2 \gamma\left(\frac{3(\alpha-2)}{4(\alpha-1)}k+\alpha\right)k \left\langle z_{k} - z_{k-1} , C \left( z_{k} \right) - C \left( z_{k-1} \right)\right\rangle\Bigg)
\nonumber\\
&\qquad- 2 \gamma\beta\left(\frac{3(\alpha-2)}{4(\alpha-1)} (k+1)+\alpha\right)(k+1)  \| C \left( z_{k+1} \right) - C \left( z_{k} \right)\|^2\nonumber\\
&\qquad+ \alpha \gamma^{2}\left(\frac{3(\alpha-2)}{4(\alpha-1)} k+\alpha\right)  \left\langle \xi_{k+1}+C \left( z_{k} \right) , C \left( z_{k} \right) - C \left( z_{k-1} \right)\right\rangle\nonumber\\
&\qquad + \gamma^{2}\left(\frac{3(\alpha-2)}{4(\alpha-1)} k+\alpha\right)(2k+\alpha)\left\langle \xi_{k+1}+C \left( z_{k} \right) -(\xi_{k}+C \left( z_{k-1} \right)), C \left( z_{k} \right) - C \left( z_{k-1} \right)\right\rangle.
\label{decbb1:inq}
\end{align}
Using Young's inequality, we obtain
\begin{align}
&\alpha \gamma^{2}\left(\frac{3(\alpha-2)}{4(\alpha-1)}k+\alpha\right)  \left\langle \xi_{k+1}+C \left( z_{k} \right) , C \left( z_{k} \right) - C \left( z_{k-1} \right)\right\rangle\nonumber\\
\leq\,\,& \frac{1}{2}\alpha \gamma^{2}\left(\frac{3(\alpha-2)}{4(\alpha-1)}k+\alpha\right)\sqrt{\frac{3(\alpha-2)}{4(\alpha-1)}k+\alpha}  \left\lVert C \left( z_{k} \right)-C \left( z_{k-1} \right) \right\rVert^2\nonumber\\
&\quad+\frac{1}{2}\alpha \gamma^{2}\sqrt{\frac{3(\alpha-2)}{4(\alpha-1)}k+\alpha}  \left\lVert \xi_{k+1}+C \left( z_{k} \right) \right\rVert^2 \nonumber\\
=\,\,& -\frac{1}{2}\alpha \gamma^{2}\Bigg(\left(\frac{3(\alpha-2)}{4(\alpha-1)}(k+1)+\alpha\right)\sqrt{\frac{3(\alpha-2)}{4(\alpha-1)}(k+1)+\alpha}  \|C \left( z_{k+1} \right)-C \left( z_{k} \right)\|^2\nonumber\\
&\qquad \qquad \quad -\left(\frac{3(\alpha-2)}{4(\alpha-1)}k+\alpha\right)\sqrt{\frac{3(\alpha-2)}{4(\alpha-1)}k+\alpha}  \left\lVert C \left( z_{k} \right)-C \left( z_{k-1} \right) \right\rVert^2\Bigg)\nonumber\\
& \quad+\frac{1}{2}\alpha \gamma^{2}\left(\frac{3(\alpha-2)}{4(\alpha-1)}(k+1)+\alpha\right)\sqrt{\frac{3(\alpha-2)}{4(\alpha-1)}(k+1)+\alpha}  \|C \left( z_{k+1} \right)-C \left( z_{k} \right)\|^2\nonumber\\
&\quad +\frac{1}{2}\alpha \gamma^{2}\sqrt{\frac{3(\alpha-2)}{4(\alpha-1)}k+\alpha}  \left\lVert \xi_{k+1}+C \left( z_{k} \right) \right\rVert^2,
\label{bbxn1}
\end{align}
and
\begin{align}
&\gamma^{2}\left(\frac{3(\alpha-2)}{4(\alpha-1)}k+\alpha\right)(2k+\alpha) \left\langle \xi_{k+1}+C \left( z_{k} \right) - (\xi_{k}+C \left( z_{k-1} \right)), C \left( z_{k} \right) - C \left( z_{k-1} \right)\right\rangle\nonumber\\
 \leq &  \ \frac{1}{2(2-\varepsilon)}\gamma^{2}\left(\frac{3(\alpha-2)}{4(\alpha-1)}k+\alpha\right)(2k+\alpha)\|\xi_{k+1}+C \left( z_{k} \right) - (\xi_{k}+C \left( z_{k-1} \right))\|^2 \nonumber\\
& \ +\frac{1}{2}(2-\varepsilon)\gamma^{2}\left(\frac{3(\alpha-2)}{4(\alpha-1)}(k+1)+\alpha\right)(2(k+1)+\alpha)\|C \left( z_{k+1} \right) - C \left( z_{k} \right)\|^2\nonumber\\
& \ -\frac{1}{2}(2-\varepsilon)\gamma^{2}\Bigg(\left(\frac{3(\alpha-2)}{4(\alpha-1)}(k+1)+\alpha\right)(2(k+1)+\alpha)\|C \left( z_{k+1} \right) - C \left( z_{k} \right)\|^2\nonumber\\
&\qquad \qquad \qquad \qquad-\left(\frac{3(\alpha-2)}{4(\alpha-1)}k+\alpha\right)(2k+\alpha)\|C \left( z_{k} \right) - C \left( z_{k-1} \right)\|^2\Bigg).
\label{bbxn2}
\end{align}
Plugging \eqref{bbxn1} and \eqref{bbxn2} into \eqref{decbb1:inq}, and adding the resulting inequality to \eqref{decwb:inq}, yields
\begin{align}
&-2 \eta \gamma(\alpha-2)\left\langle z_{k+1} - z_{*} ,\xi_{k+1}+C \left( z_{k} \right) \right\rangle\nonumber\\
&- 2 \gamma\left(\frac{3(\alpha-2)}{4(\alpha-1)}k+\alpha\right) \left( k + \alpha \right) \left\langle z_{k+1} - z_{k} , \xi_{k+1}+C \left( z_{k} \right) - (\xi_{k}+C \left( z_{k-1} \right)) \right\rangle\nonumber\\
\leq & \ 2 \gamma\Bigg(\left(\frac{3(\alpha-2)}{4(\alpha-1)}(k+1)+\alpha\right) (k+1) \left\langle z_{k+1} - z_{k} , C \left( z_{k+1} \right) - C \left( z_{k} \right)\right\rangle\nonumber\\
&\qquad \  -\left(\frac{3(\alpha-2)}{4(\alpha-1)}k+\alpha\right)k \left\langle z_{k} - z_{k-1} , C \left( z_{k} \right) - C \left( z_{k-1} \right)\right\rangle\Bigg)
\nonumber\\
& \ -\frac{1}{2}\alpha \gamma^{2}\Bigg(\left(\frac{3(\alpha-2)}{4(\alpha-1)}(k+1)+\alpha\right)\sqrt{\frac{3(\alpha-2)}{4(\alpha-1)}(k+1)+\alpha}  \|C \left( z_{k+1} \right)-C \left( z_{k} \right)\|^2\nonumber\\
&\qquad \qquad \quad -\left(\frac{3(\alpha-2)}{4(\alpha-1)}k+\alpha\right)\sqrt{\frac{3(\alpha-2)}{4(\alpha-1)}k+\alpha}  \left\lVert C \left( z_{k} \right)-C \left( z_{k-1} \right) \right\rVert^2\Bigg)\nonumber\\
& \ -\frac{1}{2}(2-\varepsilon)\gamma^{2}\Bigg(\left(\frac{3(\alpha-2)}{4(\alpha-1)}(k+1)+\alpha\right)(2(k+1)+\alpha)\|C \left( z_{k+1} \right) - C \left( z_{k} \right)\|^2\nonumber\\
&\qquad \qquad \qquad \quad -\left(\frac{3(\alpha-2)}{4(\alpha-1)}k+\alpha\right)(2k+\alpha)\|C \left( z_{k} \right) - C \left( z_{k-1} \right)\|^2\Bigg)\nonumber\\
& \ -2 \eta \gamma(\alpha-2)\left\langle z_{k+1} - z_{*} , \xi_{k+1}+C \left( z_{k+1} \right) \right\rangle +\frac{4 \eta^{2} (\alpha -2) (\alpha-1)}{3(k+1)\sqrt{k+1}} \left\lVert z_{k+1}-z_{*} \right\rVert^2\nonumber\\
& \ + \frac{1}{2}\alpha \gamma^{2}\sqrt{\frac{3(\alpha-2)}{4(\alpha-1)}k+\alpha}  \left\lVert \xi_{k+1}+C \left( z_{k} \right) \right\rVert^2 - \omega_{k} \|C \left( z_{k+1} \right)-C \left( z_{k} \right)\|^2 \nonumber\\
& \ + \frac{1}{2(2-\varepsilon)}\gamma^{2}\left(\frac{3(\alpha-2)}{4(\alpha-1)}k+\alpha\right)(2k+\alpha)\|\xi_{k+1}+C \left( z_{k} \right) - (\xi_{k}+C \left( z_{k-1} \right))\|^2,
\label{bbxn3}
\end{align}
where $\omega_{k}$ is defined in \eqref{defi:mu-k}.

Finally, after summing \eqref{bbxn3} and \eqref{dec:inquen}, we obtain
\begin{align}
& \ \E_{\eta,k+1} - \E_{\eta,k}\nonumber\\
\leq & \ 2 \gamma\Bigg(\left(\frac{3(\alpha-2)}{4(\alpha-1)}(k+1)+\alpha\right) (k+1) \left\langle z_{k+1} - z_{k} , C \left( z_{k+1} \right) - C \left( z_{k} \right)\right\rangle\nonumber\\
&\qquad \ -\left(\frac{3(\alpha-2)}{4(\alpha-1)}k+\alpha\right)k \left\langle z_{k} - z_{k-1} , C \left( z_{k} \right) - C \left( z_{k-1} \right)\right\rangle\Bigg)
\nonumber\\
& \ -\frac{1}{2}\alpha \gamma^{2}\Bigg(\left(\frac{3(\alpha-2)}{4(\alpha-1)}(k+1)+\alpha\right)\sqrt{\frac{3(\alpha-2)}{4(\alpha-1)}(k+1)+\alpha}  \|C \left( z_{k+1} \right)-C \left( z_{k} \right)\|^2\nonumber\\
&\qquad \qquad  \quad -\left(\frac{3(\alpha-2)}{4(\alpha-1)}k+\alpha\right)\sqrt{\frac{3(\alpha-2)}{4(\alpha-1)}k+\alpha}  \left\lVert C \left( z_{k} \right)-C \left( z_{k-1} \right) \right\rVert^2\Bigg)\nonumber\\
& \ -\frac{1}{2}(2-\varepsilon)\gamma^{2}\Bigg(\left(\frac{3(\alpha-2)}{4(\alpha-1)}(k+1)+\alpha\right)(2(k+1)+\alpha)\|C \left( z_{k+1} \right) - C \left( z_{k} \right)\|^2\nonumber\\
&\qquad \qquad \qquad \quad -\left(\frac{3(\alpha-2)}{4(\alpha-1)}k+\alpha\right)(2k+\alpha)\|C \left( z_{k} \right) - C \left( z_{k-1} \right)\|^2\Bigg)\nonumber\\
& \ -2 \eta \gamma(\alpha-2)\left\langle z_{k+1} - z_{*} , \xi_{k+1}+C \left( z_{k+1} \right) \right\rangle +\frac{4 \eta^{2} (\alpha -2) (\alpha-1)}{3(k+1)\sqrt{k+1}} \left\lVert z_{k+1}-z_{*} \right\rVert^2\nonumber\\
& \ + 2 \gamma \left(\nu_{0}k +\nu_{1}\right) \left\langle z_{k+1} - z_{k} , \xi_{k+1}+C \left( z_{k} \right) \right\rangle + 2\nu_{2} k \left\lVert z_{k+1} - z_{k} \right\rVert ^{2}\nonumber\\
& \ + \frac{1}{2} \gamma^{2} (\nu_{3}k +\alpha\sqrt{\frac{3(\alpha-2)}{4(\alpha-1)}k+\alpha}+\nu_{4}) \left\lVert \xi_{k+1}+C \left( z_{k} \right) \right\rVert ^{2}-\omega_{k}\|C \left( z_{k+1} \right)-C \left( z_{k} \right)\|^2  \nonumber \\
& \ - \gamma^{2}\left(\frac{3(\alpha-2)}{4(\alpha-1)}k+\alpha\right)\left(\frac{1-\varepsilon}{2-\varepsilon}k+\frac{3-2\varepsilon}{2(2-\varepsilon)}\alpha\right)  \left\lVert \xi_{k+1}+C \left( z_{k} \right) - (\xi_{k}+C \left( z_{k-1} \right)) \right\rVert ^{2}.\label{bbck1}
\end{align}
Taking into account the definition of $\F_{\eta,k}$ in \eqref{def:G}, we notice that \eqref{bbck1} is nothing else than \eqref{bbck2}.

(ii) Observe that
\begin{align}
& \ \frac{3(\alpha-2)}{2(\alpha-1)}\eta \gamma k\big\langle z_{k} - z_{*} , \xi_k +C \left( z_{k-1} \right) \big\rangle
+ \frac{3(\alpha-2)(5\alpha-2)}{32(\alpha-1)^2} \gamma^{2} k^2 \left\lVert \xi_k +C \left( z_{k-1} \right) \right\rVert ^{2}\nonumber\\
= & \ \frac{3(\alpha-2)}{2(5\alpha-2)}\left(\frac{5\alpha-2}{\alpha-1}\eta \gamma k\langle z_{k} - z_{*} , \xi_k +C \left( z_{k-1} \right) \rangle
+  \frac{(5\alpha-2)^2}{16(\alpha-1)^2}\gamma^{2} k^2  \left\lVert \xi_k +C \left( z_{k-1} \right) \right\rVert ^{2}\right)\nonumber\\
= & \ \frac{3(\alpha-2)}{2(5\alpha-2)}\left\lVert 2\eta (z_{k} - z_{*})
+ \frac{5\alpha-2}{4(\alpha-1)} \gamma k\left( \xi_k +C \left( z_{k-1} \right) \right)\right\rVert ^2-\frac{6(\alpha-2)}{5\alpha-2}\eta^2\left\lVert z_k-z_{*}\right\rVert ^2.
\label{dec:G:ine1}
\end{align}
By plugging \eqref{dec:G:ine1} into the definition of $\F_{\eta,k}$, it yields
\begin{align}
\F_{\eta,k}
= & \ \dfrac{1}{2} \left\lVert 2 \eta \left( z_{k} - z_{*} \right) + 2k \left( z_{k} - z_{k-1} \right) + \frac{5\alpha-2}{4(\alpha-1)}\gamma k \left(\xi_k +C \left( z_{k-1} \right)\right) \right\rVert ^{2} \nonumber\\
&	+ 2 \eta \left( \alpha - 1 - \eta \right) \left\lVert z_{k} - z_{*} \right\rVert ^{2}
+ 2 \eta \gamma\left(\frac{3(\alpha-2)}{4(\alpha-1)}k+\alpha\right) \big\langle z_{k} - z_{*} , \xi_k +C \left( z_{k-1} \right) \big\rangle \nonumber\\
& + \frac{1}{2}\gamma^{2}\left(\frac{3(\alpha-2)}{4(\alpha-1)}k+\alpha\right)\left( \frac{5\alpha-2}{4(\alpha-1)} k + \alpha \right)  \left\lVert \xi_k +C \left( z_{k-1} \right) \right\rVert ^{2}\nonumber\\
&-2 \gamma\left(\frac{3(\alpha-2)}{4(\alpha-1)}k+\alpha\right)k  \left\langle z_{k} - z_{k-1} , C \left( z_{k} \right) - C \left( z_{k-1} \right)\right\rangle
\nonumber\\
&+  \frac{1}{2}\gamma^{2}\left(\frac{3(\alpha-2)}{4(\alpha-1)}k+\alpha\right)\left((2-\varepsilon)(2k+\alpha)+\alpha\sqrt{\frac{3(\alpha-2)}{4(\alpha-1)}k+\alpha} \right)\left\lVert C \left( z_{k} \right)-C \left( z_{k-1} \right) \right\rVert^2\nonumber\\
= & \ \dfrac{1}{2} \left\lVert 2 \eta \left( z_{k} - z_{*} \right) + 2k \left( z_{k} - z_{k-1} \right) + \frac{5\alpha-2}{4(\alpha-1)}\gamma k \left(\xi_k +C \left( z_{k-1} \right)\right) \right\rVert ^{2} \nonumber\\
&+\frac{3(\alpha-2)}{2(5\alpha-2)} \left\lVert 2 \eta \left( z_{k} - z_{*} \right) + \frac{5\alpha-2}{4(\alpha-1)}\gamma k \left(\xi_k +C \left( z_{k-1} \right)\right) \right\rVert ^{2} \nonumber\\
&	+ 2 \eta(\alpha-1) \left( 1 - \frac{8\eta}{5\alpha-2} \right) \left\lVert z_{k} - z_{*} \right\rVert ^{2}
+ \frac{1}{2} \gamma^{2}\alpha (2k+\alpha) \left\lVert \xi_k +C \left( z_{k-1} \right) \right\rVert ^{2}\nonumber\\
&-2 \gamma\left(\frac{3(\alpha-2)}{4(\alpha-1)}k+\alpha\right)k  \left\langle z_{k} - z_{k-1} , C \left( z_{k} \right) - C \left( z_{k-1} \right)\right\rangle+2\eta \alpha \gamma \big\langle z_{k} - z_{*} , \xi_k +C \left( z_{k-1} \right) \big\rangle
\nonumber\\
&+\frac{1}{2}\gamma^{2}\left(\frac{3(\alpha-2)}{4(\alpha-1)}k+\alpha\right)\left((2-\varepsilon)(2k+\alpha)+\alpha\sqrt{\frac{3(\alpha-2)}{4(\alpha-1)}k+\alpha} \right)  \left\lVert C \left( z_{k} \right)-C \left( z_{k-1} \right) \right\rVert^2\nonumber\\
= & \ \frac{\alpha+2}{5\alpha-2} \left\lVert 2 \eta \left( z_{k} - z_{*} \right) + 2k \left( z_{k} - z_{k-1} \right) + \frac{5\alpha-2}{4(\alpha-1)}\gamma k\left(\xi_k +C \left( z_{k-1} \right)\right) \right\rVert ^{2} \nonumber\\
&+\frac{3(\alpha-2)}{4(5\alpha-2)} \left\lVert 4 \eta \left( z_{k} - z_{*} \right)+ 2k \left( z_{k} - z_{k-1} \right) + \frac{5\alpha-2}{2(\alpha-1)}\gamma k \left(\xi_k +C \left( z_{k-1} \right)\right) \right\rVert ^{2} \nonumber\\
&	+ 2 \eta(\alpha-1) \left( 1 - \frac{8\eta}{5\alpha-2} \right) \left\lVert z_{k} - z_{*} \right\rVert ^{2}
+ \frac{1}{2} \gamma^{2}\alpha (2k+\alpha) \left\lVert \xi_k +C \left( z_{k-1} \right) \right\rVert ^{2}\nonumber\\
&+\frac{3(\alpha-2)}{5\alpha-2} k^2 \left\lVert z_k-z_{k-1} \right\rVert ^{2}-2 \gamma\left(\frac{3(\alpha-2)}{4(\alpha-1)}k+\alpha\right)k   \left\langle z_{k} - z_{k-1} , C \left( z_{k} \right) - C \left( z_{k-1} \right)\right\rangle
\nonumber\\
&+\frac{1}{2}\gamma^{2}\left(\frac{3(\alpha-2)}{4(\alpha-1)}k+\alpha\right)\left((2-\varepsilon)(2k+\alpha)+\alpha\sqrt{\frac{3(\alpha-2)}{4(\alpha-1)}k+\alpha} \right)  \left\lVert C \left( z_{k} \right)-C \left( z_{k-1} \right) \right\rVert^2\nonumber\\
&+2\eta \alpha \gamma \big\langle z_{k} - z_{*} , \xi_k +C \left( z_{k-1} \right) \big\rangle,
\label{def:GG}
\end{align}
where for obtaining the last statement we used the identity
\begin{equation*}
\|x\|^2+\|y\|^2=\frac{1}{2}\left(\|x+y\|^2+\|x-y\|^2\right) \quad \forall x,y\in \mathcal{H}.
\end{equation*}
Finally, by employing in (\ref{def:GG})
\begin{align*}
& -\frac{3(\alpha-2)}{2(\alpha-1)}\gamma k^{2}  \left\langle z_{k} - z_{k-1} , C \left( z_{k} \right) - C \left( z_{k-1} \right)\right\rangle
+\frac{3(\alpha-2)}{4(\alpha-1)}(2-\varepsilon) \gamma^{2} k^2\left\lVert C \left( z_{k} \right)-C \left( z_{k-1} \right) \right\rVert^2\nonumber\\
= & \ \frac{3(\alpha-2)}{4(\alpha-1)}\gamma k^{2}  \left\lVert \frac{1}{\sqrt{(2-\varepsilon) \gamma}}(z _{k} - z_{k-1} )-\sqrt{(2-\varepsilon) \gamma}( C \left( z_{k} \right) - C \left( z_{k-1} \right))\right\rVert ^2\nonumber\\
& -\frac{3(\alpha-2)}{4(\alpha-1)(2-\varepsilon)} k^2\left\lVert z_k-z_{k-1} \right\rVert^2,
\end{align*}
and then the Lipschitz continuity of $C$ , it yields
\begin{align*}
\F_{\eta,k}
= & \ \frac{\alpha+2}{5\alpha-2} \left\lVert 2 \eta \left( z_{k} - z_{*} \right) + 2k \left( z_{k} - z_{k-1} \right) + \frac{5\alpha-2}{4(\alpha-1)}\gamma k \left(\xi_k +C \left( z_{k-1} \right)\right) \right\rVert ^{2} \nonumber\\
&+\frac{3(\alpha-2)}{4(5\alpha-2)} \left\lVert 4 \eta \left( z_{k} - z_{*} \right)+ 2k \left( z_{k} - z_{k-1} \right) + \frac{5\alpha-2}{2(\alpha-1)}\gamma k \left(\xi_k +C \left( z_{k-1} \right)\right) \right\rVert ^{2} \nonumber\\
&	+ 2 \eta(\alpha-1) \left( 1 - \frac{8\eta}{5\alpha-2} \right) \left\lVert z_{k} - z_{*} \right\rVert ^{2}
+ \frac{1}{2} \gamma^{2}\alpha (2k+\alpha) \left\lVert \xi_k +C \left( z_{k-1} \right) \right\rVert ^{2}\nonumber\\
&+3(\alpha-2)\left(\frac{1}{5\alpha-2}-\frac{1}{4(2-\varepsilon)(\alpha-1)}\right) k^2 \left\lVert z_k-z_{k-1} \right\rVert ^{2}\nonumber\\
& +\frac{3(\alpha-2)}{4(\alpha-1)}\gamma k^{2}  \left\lVert \frac{1}{\sqrt{(2-\varepsilon) \gamma}}(z _{k} - z_{k-1} )-\sqrt{(2-\varepsilon) \gamma}( C \left( z_{k} \right) - C \left( z_{k-1} \right))\right\rVert ^2\nonumber\\
&-2 \alpha \gamma k   \left\langle z_{k} - z_{k-1} , C \left( z_{k} \right) - C \left( z_{k-1} \right)\right\rangle
+\rho_k  \left\lVert C \left( z_{k} \right)-C \left( z_{k-1} \right) \right\rVert^2\nonumber\\
&+2\eta \alpha \gamma \big\langle z_{k} - z_{*} , \xi_k +C \left( z_{k-1} \right) \big\rangle\nonumber\\
\geq & \ \frac{\alpha+2}{5\alpha-2} \left\lVert 2 \eta \left( z_{k} - z_{*} \right) + 2k \left( z_{k} - z_{k-1} \right) + \frac{5\alpha-2}{4(\alpha-1)}\gamma k \left(\xi_k +C \left( z_{k-1} \right)\right) \right\rVert ^{2} \nonumber\\
&+\frac{3(\alpha-2)}{4(5\alpha-2)} \left\lVert 4 \eta \left( z_{k} - z_{*} \right)+ 2k \left( z_{k} - z_{k-1} \right) + \frac{5\alpha-2}{2(\alpha-1)}\gamma k \left(\xi_k +C \left( z_{k-1} \right)\right) \right\rVert ^{2} \nonumber\\
&	+ 2 \eta(\alpha-1) \left( 1 - \frac{8\eta}{5\alpha-2} \right) \left\lVert z_{k} - z_{*} \right\rVert ^{2}
+ \frac{1}{2} \gamma^{2}\alpha (2k+\alpha) \left\lVert \xi_k +C \left( z_{k-1} \right) \right\rVert ^{2}\nonumber\\
&+\left(3(\alpha-2)\left(\frac{1}{5\alpha-2}-\frac{1}{4(2-\varepsilon)(\alpha-1)}\right) k^2-\frac{2}{\beta}\alpha \gamma k\right) \left\lVert z_k-z_{k-1} \right\rVert ^{2}\nonumber\\
& +\frac{3(\alpha-2)}{4(\alpha-1)}\gamma k^{2}  \left\lVert \frac{1}{\sqrt{(2-\varepsilon) \gamma}}(z _{k} - z_{k-1} )-\sqrt{(2-\varepsilon) \gamma}( C \left( z_{k} \right) - C \left( z_{k-1} \right))\right\rVert ^2\nonumber\\
&+2\eta \alpha \gamma \big\langle z_{k} - z_{*} , \xi_k +C \left( z_{k-1} \right) \big\rangle
+\rho_{k}  \left\lVert C \left( z_{k} \right)-C \left( z_{k-1} \right) \right\rVert^2,
\end{align*}
where
\begin{equation}
\rho_{k}:=\frac{1}{2}\gamma^{2}\left(\left(\frac{3(\alpha-2)}{4(\alpha-1)}k+\alpha\right)\alpha\sqrt{\frac{3(\alpha-2)}{4(\alpha-1)}k+\alpha}+\frac{11\alpha-14}{4(\alpha-1)}(2-\varepsilon)\alpha k+(2-\varepsilon)\alpha^2\right)>0.\nonumber
\end{equation}
The desired inequality follows by taking into account that the last term is nonnegative.
\end{prooffff}

\begin{proofffff}
(i) Since $0 < \gamma<\left(1+\frac{\varepsilon}{2-\varepsilon}\right)\beta$ and $0<\varepsilon<\frac{3(\alpha-2)}{4(\alpha-1)}$, it holds
\begin{equation}
2\beta -(2-\varepsilon)\gamma > 0 \quad \textnormal{and} \quad \frac{1}{5\alpha-2}-\frac{1}{4(\alpha-1)(2-\varepsilon)}>0.\nonumber
\end{equation}
Therefore, for $k$ large enough
\begin{equation}\label{deciss:G}
\omega_{k}>0 \quad \textnormal{and} \quad \frac{3}{2}(\alpha-2)\left(\frac{1}{5\alpha-2}-\frac{1}{4(2-\varepsilon)(\alpha-1)}\right) k^2-\frac{2}{\beta}\alpha \gamma k >0.
\end{equation}
For every $k \geq 1$, the discriminant of the quadratic expression in $S_k$ reads
\begin{align}
\Delta_k& =(2\eta \alpha \gamma )^2-2\eta(\alpha-1)\left(1-\frac{8}{5\alpha-2}\eta\right)\alpha \gamma^{2} k\nonumber\\
&=-2\eta\alpha \gamma^{2}(\alpha-1)\left(1-\frac{8}{5\alpha-2}\eta\right) k+4\eta^2\alpha^2 \gamma^{2}.\nonumber
\end{align}
Since $0<\eta<\frac{5\alpha-2}{8}$, the coefficient of $k$ is negative, therefore, for $k$ large enough, $\Delta_k \leq 0$, which, invoking \cref{lem111} (i),  implies $S_k\geq 0$. Consequently, there exists $K_0(\eta) \geq 1$ such that for every $k \geq K_0(\eta)$ the inequalities in \eqref{deciss:G} hold and $S_k \geq 0$. By taking into consideration the lower bound of $\F_{\eta,k}$ in \eqref{G:inequa}, we see that \eqref{def:S} holds for every $k \geq K_{0}(\eta)$.

(ii) Let $k \geq 1$ be fixed. For the discriminant of the quadratic expression in $R_k$, it holds
\begin{align}
\frac{\Delta_k}{4\gamma^{2}}&=(\nu_{0}k+\nu_{1})^2-\frac{9\alpha-2}{2(5\alpha-2)}\nu_{2}\left(\nu_{3} k+\alpha \sqrt{\nu_{5}k+\alpha}+\nu_{4} \right)k\nonumber\\
&=\left(\nu_{0}^2-\frac{9\alpha-2}{2(5\alpha-2)}\nu_{2}\nu_{3}\right)k^2-\frac{9\alpha-2}{2(5\alpha-2)}\nu_{2}\alpha k\sqrt{\nu_{5}k+\alpha}
+\left(2\nu_{0}\nu_{1}-\frac{9\alpha-2}{2(5\alpha-2)}\nu_{2}\nu_{4}\right)k+\nu_{1}^2.\nonumber
\end{align}
Since $\left(\nu_{0}^2-\frac{9\alpha-2}{2(5\alpha-2)}\nu_{2}\nu_{3}\right)k^2$ is the dominant term in the above polynomial, it suffices to show that for every $\eta$ in an open interval it holds
\begin{equation}\label{zjsjw1}
\nu_{0}^2-\frac{9\alpha-2}{2(5\alpha-2)}\nu_{2}\nu_{3}<0,
\end{equation}
where
\begin{align}
\nu_{0}&=2 (\eta+1-\alpha)-\frac{1}{4(\alpha-1)}(\alpha+2)(\alpha-2),\nonumber\\
\nu_{2}\nu_{3}&=-\frac{1}{\alpha-1}(5\alpha-2)(\alpha-2)(\eta+1-\alpha).\nonumber
\end{align}
This will ensure the existence of some integer $K_{1}(\eta)\geq 1$ such that $\Delta_k\leq0$ for every $k\geq  K_{1}(\eta)$, and further, due to \cref{lem111} (ii), that $R_k\leq 0$ for every $k\geq  K_{1}(\eta)$.

We set $\zeta:=\eta+1-\alpha\leq 0$ and get
\begin{align}
\nu_{0}&=2\zeta-\frac{1}{4(\alpha-1)}(\alpha+2)(\alpha-2),\nonumber\\
\nu_{2}\nu_{3}&=-\frac{1}{\alpha-1}(5\alpha-2)(\alpha-2)\zeta,\nonumber
\end{align}
and
\begin{align}\label{zjsjw2}
\nu_{0}^2-\frac{9\alpha-2}{2(5\alpha-2)}\nu_{2}\nu_{3} = & \ \left(2\zeta -\frac{(\alpha+2)(\alpha-2)}{4(\alpha-1)}\right)^2+\frac{9\alpha-2}{2(5\alpha-2)}\frac{(5\alpha-2)(\alpha-2)}{\alpha-1}\zeta\nonumber\\
= & \ 4\zeta^2+\frac{(\alpha-2)(7\alpha-6)}{2(\alpha-1)}\zeta+\frac{(\alpha+2)^2(\alpha-2)^2}{16(\alpha-1)^2}.
\end{align}
The discriminant of this expression is
\begin{align}
\Delta_\zeta =  \frac{(\alpha-2)^2(7\alpha-6)^2}{4(\alpha-1)^2}-\frac{(\alpha+2)^2(\alpha-2)^2}{(\alpha-1)^2}
= \frac{1}{4(\alpha-1)^2}(\alpha-2)^2(9\alpha-2)(5\alpha-10)>0.\nonumber
\end{align}
Hence, in order to guarantee (\ref{zjsjw1}), we have to choose $\zeta$ between the two roots of the quadratic function arising in \eqref{zjsjw2}, in other words,
\begin{align}
\zeta_1(\alpha)&:= \frac{1}{8}\left(-\frac{1}{2(\alpha-1)}(\alpha-2)(7\alpha-6)-\sqrt{\Delta_\zeta}\right)\nonumber\\
&=-\frac{1}{16(\alpha-1)}(\alpha-2)\left(7\alpha-6+\sqrt{(9\alpha-2)(5\alpha-10)}\right)\nonumber\\
< \zeta & =\eta+1-\alpha< \nonumber\\
\zeta_2(\alpha)&:= \frac{1}{8}\left(-\frac{1}{2(\alpha-1)}(\alpha-2)(7\alpha-6)+\sqrt{\Delta_\zeta}\right)\nonumber\\
&=-\frac{1}{16(\alpha-1)}(\alpha-2)\left(7\alpha-6-\sqrt{(9\alpha-2)(5\alpha-10)}\right).\nonumber
\end{align}
Obviously $\zeta_1(\alpha)<0$ and from Vieta formula $\zeta_1(\alpha)\cdot \zeta_2(\alpha)=\frac{(\alpha+2)^2(\alpha-2)^2}{64(\alpha-1)^2}>0$, it follows that $\zeta_2(\alpha)<0$.
Therefore, in order to be sure that $\nu_{0}^2-\frac{9\alpha-2}{2(5\alpha-2)}\nu_{2}\nu_{3}<0$,  $\eta$ must be chosen such that
\begin{equation}
\alpha-1+\zeta_1(\alpha)<\eta<\alpha-1+\zeta_2(\alpha). \nonumber
\end{equation}

Next we will show that
\begin{equation}
0<\alpha-1-\frac{1}{16(\alpha-1)}(\alpha-2)(7\alpha-6)<\frac{5\alpha}{8}-\frac{1}{4}.
\label{left:hand}
\end{equation}
Indeed, the first inequality follows immediately, since
\begin{align}
\alpha-1-\frac{1}{16(\alpha-1)}(\alpha-2)(7\alpha-6)&=\frac{1}{16(\alpha-1)}\left(16(\alpha-1)^2-(\alpha-2)(7\alpha-6)\right)\nonumber\\
&=\frac{1}{16(\alpha-1)}\left(\alpha^2+4(\alpha-1)(2\alpha-1)\right)>0.\nonumber
\end{align}
Using this relation, the second inequality in (\ref{left:hand}) can be equivalently written as
\begin{equation}
\alpha^2+4(\alpha-1)(2\alpha-1)<2(\alpha-1)(5\alpha-2)\ \Leftrightarrow \ 0<\alpha(\alpha-2),\nonumber
\end{equation}
which is true as $\alpha>2$.

From (\ref{left:hand}) we immediately deduce that
\begin{equation*}
0<\alpha-1+\zeta_2(\alpha) \qquad \textrm{ and } \qquad \alpha-1+\zeta_1(\alpha)< \frac{5\alpha}{8}-\frac{1}{4}.
\end{equation*}
This allows us to define 
\begin{align}
\underline{\eta}(\alpha)&:= \max\left\{0,\alpha-1+\zeta_1(\alpha)\right\}\nonumber\\ &=\max\left\{0,\frac{1}{16(\alpha-1)}\alpha^2+\frac{1}{4}(2\alpha-1)-\frac{1}{16(\alpha-1)}(\alpha-2)\sqrt{(9\alpha-2)(5\alpha-10)}\right\};\nonumber\\
\overline{\eta}(\alpha)&:= \min\left\{\frac{5}{8}\alpha-\frac{1}{4},\alpha-1+\zeta_2(\alpha)\right\}\nonumber\\ &=\min\left\{\frac{5}{8}\alpha-\frac{1}{4},\frac{1}{16(\alpha-1)}\alpha^2+\frac{1}{4}(2\alpha-1)+\frac{1}{16(\alpha-1)}(\alpha-2)\sqrt{(9\alpha-2)(5\alpha-10)}\right\},\nonumber
\end{align}
and to notice that $\underline{\eta}(\alpha)<\overline{\eta}(\alpha)$. In conclusion, choosing $\eta$ such that $\underline{\eta}(\alpha)<\eta<\overline{\eta}(\alpha)$, we have
\begin{equation*}
\nu_{0}^2-\frac{9\alpha-2}{2(5\alpha-2)}\nu_{2}\nu_{3}<0,
\end{equation*}
and therefore there exists $K_{1}(\eta)\geq 1$ such that $R_k\leq 0$ for every $k\geq K_{1}(\eta)$.
\end{proofffff}

\bibliographystyle{plain}
\bibliography{reference-en} %

\end{document}